\documentclass[12pt]{article}

\usepackage{amsfonts,amsmath,amsthm,amssymb,amscd}
\usepackage[dvips]{graphicx}
\usepackage{mathrsfs}

\theoremstyle{definition}
\newtheorem{theorem}{Theorem}
\newtheorem{lemma}[theorem]{Lemma}
\newtheorem{proposition}[theorem]{Proposition}
\newtheorem{corollary}[theorem]{Corollary}

\numberwithin{equation}{section}
\numberwithin{theorem}{section}

\begin{document}

\begin{center}
{\bf{\Large Differential equations \\ satisfied by modular forms of level 5 }}
\end{center}

\begin{center}
By Kazuhide Matsuda
\end{center}

\begin{center}
Faculty of Fundamental Science, National Institute of Technology, Niihama College,\\
7-1 Yagumo-chou, Niihama, Ehime 792-8580, Japan \\
E-mail: matsuda@sci.niihama-nct.ac.jp  \\
Fax: 81-0897-37-7809 
\end{center}

\noindent
{\bf Abstract}
This paper describes the derivation of the level 5 versions of Ramanujan's system of ordinary differential equations satisfied by 
the Eisenstein series, $E_2(q),E_4(q)$, and $E_6(q).$

\section{Introduction}
\label{intro}
Let $\mathbb{N}_0$ and $\mathbb{N},$ 
denote sets of nonnegative and positive integers, respectively. 
For $k$ and $n\in\mathbb{N},$ 
$\sigma_k(n)$ is the sum of the $k$-th power of the positive divisors of $n,$ 
and 
$\sigma_k(n)=0$ for $n\in\mathbb{Q}\setminus\mathbb{N}_0.$ 
\par
The {\it upper half plane} $\mathbb{H}^2$ is defined by 
$
\mathbb{H}^2=
\{
\tau\in\mathbb{C} \,\, | \,\, \Im \tau>0
\}. 
$
Throughout this paper, we set $q=\exp(2\pi i  \tau)$ and define the {\it Dedekind eta function} as  
$
\displaystyle
\eta(\tau)=q^{\frac{1}{24}} \prod_{n=1}^{\infty} (1-q^n). 
$
\par
In \cite{Darboux}, Darboux studied a problem in mechanics and differential geometry, and encountered the system of ordinary differential equations (ODEs), 
\begin{equation}
\label{eqn:Halphen}
x\frac{d}{dx}(u_1+u_2)=u_1 u_2,  \,\,
x\frac{d}{dx}(u_1+u_3)=u_1 u_3,  \,\,
x\frac{d}{dx}(u_2+u_3)=u_2 u_3,  \,\,x=\exp(\pi i \tau).
\end{equation}
In \cite[pp. 330]{Halphen}, Halphen provided the following solution to the problem studied by Darboux, 
\begin{align*}
u_1=&1+8\sum_{n=1}^{\infty}\frac{x^{2n}}{(1+x^{2n})^2}, \,\,
u_2=-8\sum_{n=1}^{\infty}\frac{x^{2n-1}}{(1-x^{2n-1})^2}, \,\,
u_3=8\sum_{n=1}^{\infty}\frac{x^{2n-1}}{(1+x^{2n-1})^2}. 
\end{align*}
Equation (\ref{eqn:Halphen}) is called Halphen's system. 
\par
For studying Painlev\'e type equations of the third order, Chazy \cite{Chazy} considered the nonlinear differential equation for the complex function 
\begin{equation}
\label{eqn:Chazy}
y^{\prime \prime \prime}=
2y y^{\prime \prime}-3(y^{\prime})^2. 
\end{equation}
Note that $y=u_1+u_2+u_3=\pi i E_2(\tau)$ is a solution of Chazy's equation (\ref{eqn:Chazy}), 
where 
the Eisenstein series $E_2, E_4$, and $E_6$ are respectively defined by 
\begin{align*}
E_2(q)=E_2(\tau)&:=1-24\sum_{n=1}^{\infty} \sigma_1(n) q^n, \,\,
E_4(q)=E_4(\tau):=1+240\sum_{n=1}^{\infty} \sigma_3(n) q^n, \\
E_6(q)=E_4(\tau)&:=1-504\sum_{n=1}^{\infty} \sigma_5(n) q^n.
\end{align*}
\par
In \cite{Ramanujan}, 
Ramanujan derived the following system of ODEs, 
\begin{equation}
\label{eqn:Ramanujan-ODE}
q\frac{d E_2}{dq}=\frac{(E_2)^2-E_4  }{12  }, \,\,
q\frac{d E_4}{dq}=\frac{E_2 E_4-E_6  }{3  }, \,\,
q\frac{d E_6}{dq}=\frac{E_2 E_6-(E_4)^2  }{2  }. 
\end{equation}
In particular, 
by eliminating $E_4$ and $E_6$, $y(\tau)=\pi i E_2(\tau)$ is observed to be a solution to Chazy's equation (\ref{eqn:Chazy}). 
Further, Ohyama \cite{Ohyama0} showed that 
Halphen's differential field is an extension of Ramanujan's differential field, 
whose Galois group is the symmetric group $S_3.$
\par
In \cite{Huber} and \cite{Matsuda}, ODEs satisfied by the cubic theta functions were derived as
\begin{align*}
a(q)=&\sum_{m,n\in\mathbb{Z}} q^{m^2+mn+n^2}, \,\,
b(q)=\sum_{m,n\in\mathbb{Z}} \omega^{n-m} q^{m^2+mn+n^2}, \\
c(q)=&\sum_{m,n\in\mathbb{Z}} q^{(n+\frac13)^2+(n+\frac13)(m+\frac13)+(m+\frac13)^2}, \,\,\omega=e^{\frac{2\pi i}{3}}, \,\,|q|<1. 
\end{align*}
Therefore, 
Huber used Ramanujan's theory of theta functions. 
In addition, 
Cooper \cite{Cooper} also treated ODEs satisfied by modular forms by using Ramanujan's theory. 
In the present study, the theory of theta functions with rational characteristics has been adopted.  
Noted is that Hermite \cite{Hermite} introduced theta functions with characteristics, and 
Farkas and Kra \cite{Farkas-Kra} developed the theory of theta functions with rational characteristics. 
\par
The aim of the current research was to derive the ODEs satisfied by the modular forms of  level 5, 
and to obtain an extension of Ramanujan's differential field by using the modular forms of  level 5.

\par
The main theorems of this study are as follows.

\begin{theorem}
\label{thm-level5-E_2(q)}
{\it
For $q\in\mathbb{C}$ with $|q|<1,$ 
set 
\begin{align*}
P(q)=&\cot \frac{2 \pi}{5} + 4 \sin \frac{4 \pi}{5} \sum_{n=1}^{\infty} (d_{1,5}(n)-d_{4,5}(n)) q^n
-4\sin \frac{2\pi}{5} \sum_{n=1}^{\infty} (d_{2,5}(n)-d_{3,5}(n)) q^n, \\
Q(q)=&\cot \frac{ \pi}{5} + 4 \sin \frac{2 \pi}{5} \sum_{n=1}^{\infty} (d_{1,5}(n)-d_{4,5}(n)) q^n
+4\sin \frac{4\pi}{5} \sum_{n=1}^{\infty} (d_{2,5}(n)-d_{3,5}(n)) q^n, \\
R(q)=&E_2(q)=1-24\sum_{n=1}^{\infty} \sigma_1(n) q^n. 
\end{align*}
Then, we have 
\begin{align*}
\label{eqn-level5-E_2(q)}
q\frac{d}{dq}P=&\frac{-13P^3-39P^2Q+47PQ^2-9Q^3+2PR}{24},  \,\,
q\frac{d}{dq}Q=\frac{9P^3+47P^2Q+39PQ^2-13Q^3+2QR}{24}, \\
q\frac{d}{dq}R=&\frac{5P^4-15P^3Q-155P^2Q^2+15PQ^3+5Q^4+R^2}{12}. 
\end{align*}
}
\end{theorem}


\begin{theorem}
\label{thm-level5-E_2(q^5)}
{\it
Further, for $q\in\mathbb{C}$ with $|q|<1,$ set
\begin{align*}
P(q)=&1+ 10 \sum_{n=1}^{\infty} (d_{2,5}(n)-d_{3,5}(n)) q^n, \,\,
Q(q)=3 + 10\sum_{n=1}^{\infty} (d_{1,5}(n)-d_{4,5}(n)) q^n, \\
R(q)=&E_2(q^5)=1-24\sum_{n=1}^{\infty} \sigma_1(n) q^{5n}. 
\end{align*}
Then, we have 
\begin{align*}
q\frac{d}{dq}P=&\frac{13P^3+39P^2Q-47PQ^2+9Q^3+50PR}{120},  \,\,
q\frac{d}{dq}Q=\frac{-9P^3-47P^2Q-39PQ^2+13Q^3+50QR}{120}, \\
q\frac{d}{dq}R=&\frac{P^4-3P^3Q-31P^2Q^2+3PQ^3+Q^4+125R^2}{300}. 
\end{align*}
}
\end{theorem}

\par
The remainder of this paper is organized as follows.
Section \ref{sec:properties} provides an overview of Farkas' and Kra's theory of theta functions with rational characteristics. 
Next, Section \ref{sec:Weierstrass} discusses Weierstrass elliptic function theory, and 
Section \ref{sec:proof-preliminary} describes the preliminary results obtained for $P, Q,$ and $R.$ 
Further, Section \ref{sec:application-residue-theorem} derives the identities of the theta derivatives. 
Sections \ref{sec:proof:level5-E_2(q)} and \ref{sec:proof:level5-E_2(q^5)} provide the proofs for Theorems \ref{thm-level5-E_2(q)} and \ref{thm-level5-E_2(q^5)}, respectively. 
Section \ref{sec:product-series-level5} presents the eta products, $\eta^5(\tau)/\eta(5\tau)$ and $\eta^5(5\tau)/\eta(\tau).$
Section \ref{sec:theorem-Farkas-Kra} discusses the theorem of Farkas and Kra stating the theta derivatives. 
Finally, Section \ref{sec:Riccati-level-5} derives the Riccati equations satisfied by the modular forms of level 5.

\subsection*{Acknowledgments}
This work was supported by JSPS KAKENHI Grant Number JP17K14213.

\section{Properties of the theta functions}
\label{sec:properties}

\subsection{Definitions}
Following the work of Farkas and Kra \cite{Farkas-Kra}, 
we introduce the {\it theta function with characteristics,} 
which is defined by 
\begin{align*}
\theta 
\left[
\begin{array}{c}
\epsilon \\
\epsilon^{\prime}
\end{array}
\right] (\zeta, \tau) 
=
\theta 
\left[
\begin{array}{c}
\epsilon \\
\epsilon^{\prime}
\end{array}
\right] (\zeta) 
:=&\sum_{n\in\mathbb{Z}} \exp
\left(2\pi i\left[ \frac12\left(n+\frac{\epsilon}{2}\right)^2 \tau+\left(n+\frac{\epsilon}{2}\right)\left(\zeta+\frac{\epsilon^{\prime}}{2}\right) \right] \right), 
\end{align*}
where $\epsilon, \epsilon^{\prime}\in\mathbb{R}, \, \zeta\in\mathbb{C},$ and $\tau\in\mathbb{H}^{2}.$ 
\par
The relation between  theta functions with rational characteristics and Jacobi theta functions is given by 
\begin{equation*}
\vartheta_1(\pi \zeta)
=
-
\theta 
\left[
\begin{array}{c}
1 \\
1
\end{array}
\right] (\zeta), \,\, 
\vartheta_2(\pi \zeta)
=
\theta 
\left[
\begin{array}{c}
1 \\
0
\end{array}
\right] (\zeta), \,\, 
\vartheta_3(\pi \zeta)
=
\theta 
\left[
\begin{array}{c}
0 \\
0
\end{array}
\right] (\zeta), \,\, 
\vartheta_4(\pi \zeta)
=
\theta 
\left[
\begin{array}{c}
0 \\
1
\end{array}
\right] (\zeta). 
\end{equation*}
\par
The {\it theta constants} are given by 
\begin{equation*}
\theta 
\left[
\begin{array}{c}
\epsilon \\
\epsilon^{\prime}
\end{array}
\right]
:=
\theta 
\left[
\begin{array}{c}
\epsilon \\
\epsilon^{\prime}
\end{array}
\right] (0, \tau).
\end{equation*}
\par
Furthermore, 
we denote the {\it theta derivatives} by 
\begin{equation*}
\theta^{\prime} 
\left[
\begin{array}{c}
\epsilon \\
\epsilon^{\prime}
\end{array}
\right]
:=\left.
\frac{\partial}{\partial \zeta} 
\theta 
\left[
\begin{array}{c}
\epsilon \\
\epsilon^{\prime}
\end{array}
\right] (\zeta, \tau)
\right|_{\zeta=0}, 
\,
\theta^{\prime \prime} 
\left[
\begin{array}{c}
\epsilon \\
\epsilon^{\prime}
\end{array}
\right]
:=\left.
\frac{\partial^2 }{\partial \zeta^2} 
\theta 
\left[
\begin{array}{c}
\epsilon \\
\epsilon^{\prime}
\end{array}
\right] (\zeta, \tau)
\right|_{\zeta=0},
\end{equation*}
and
\begin{equation*}
\theta^{\prime \prime \prime} 
\left[
\begin{array}{c}
\epsilon \\
\epsilon^{\prime}
\end{array}
\right]
:=\left.
\frac{\partial^3 }{\partial \zeta^3} 
\theta 
\left[
\begin{array}{c}
\epsilon \\
\epsilon^{\prime}
\end{array}
\right] (\zeta, \tau)
\right|_{\zeta=0},\,\,
\theta^{(n)} 
\left[
\begin{array}{c}
\epsilon \\
\epsilon^{\prime}
\end{array}
\right]
:=\left.
\frac{\partial^n}{\partial \zeta^n} 
\theta 
\left[
\begin{array}{c}
\epsilon \\
\epsilon^{\prime}
\end{array}
\right] (\zeta, \tau)
\right|_{\zeta=0}, \,\,\,(n=1,2,3,4,\ldots). 
\end{equation*}
In particular, Jacobi's derivative formula is given by 
\begin{equation}
\label{eqn:Jacobi-derivative}
\theta^{\prime} 
\left[
\begin{array}{c}
1 \\
1
\end{array}
\right] 
=
-\pi 
\theta
\left[
\begin{array}{c}
0 \\
0
\end{array}
\right] 
\theta
\left[
\begin{array}{c}
1 \\
0
\end{array}
\right] 
\theta
\left[
\begin{array}{c}
0 \\
1
\end{array}
\right].  
\end{equation}

\subsection{Basic properties}
First, we note that 
for $m,n\in\mathbb{Z},$ 
\begin{equation}
\label{eqn:integer-char}
\theta 
\left[
\begin{array}{c}
\epsilon \\
\epsilon^{\prime}
\end{array}
\right] (\zeta+n+m\tau, \tau) =
\exp(2\pi i)\left[\frac{n\epsilon-m\epsilon^{\prime}}{2}-m\zeta-\frac{m^2\tau}{2}\right]
\theta 
\left[
\begin{array}{c}
\epsilon \\
\epsilon^{\prime}
\end{array}
\right] (\zeta,\tau),
\end{equation}
and 
\begin{equation}
\theta 
\left[
\begin{array}{c}
\epsilon +2m\\
\epsilon^{\prime}+2n
\end{array}
\right] 
(\zeta,\tau)
=\exp(\pi i \epsilon n)
\theta 
\left[
\begin{array}{c}
\epsilon \\
\epsilon^{\prime}
\end{array}
\right] 
(\zeta,\tau).
\end{equation}
Furthermore, 
it is easy to see that 
\begin{equation*}
\theta 
\left[
\begin{array}{c}
-\epsilon \\
-\epsilon^{\prime}
\end{array}
\right] (\zeta,\tau)
=
\theta 
\left[
\begin{array}{c}
\epsilon \\
\epsilon^{\prime}
\end{array}
\right] (-\zeta,\tau)
\,\,
\mathrm{and}
\,\,
\theta^{\prime} 
\left[
\begin{array}{c}
-\epsilon \\
-\epsilon^{\prime}
\end{array}
\right] (\zeta,\tau)
=
-
\theta^{\prime} 
\left[
\begin{array}{c}
\epsilon \\
\epsilon^{\prime}
\end{array}
\right] (-\zeta,\tau).
\end{equation*}
\par
For $m,n\in\mathbb{R},$ 
we get 
\begin{align}
\label{eqn:real-char}
&\theta 
\left[
\begin{array}{c}
\epsilon \\
\epsilon^{\prime}
\end{array}
\right] \left(\zeta+\frac{n+m\tau}{2}, \tau\right)   \notag\\
&=
\exp(2\pi i)\left[
-\frac{m\zeta}{2}-\frac{m^2\tau}{8}-\frac{m(\epsilon^{\prime}+n)}{4}
\right]
\theta 
\left[
\begin{array}{c}
\epsilon+m \\
\epsilon^{\prime}+n
\end{array}
\right] 
(\zeta,\tau). 
\end{align}
Note that 
$\theta 
\left[
\begin{array}{c}
\epsilon \\
\epsilon^{\prime}
\end{array}
\right] \left(\zeta, \tau\right)$ has only one zero in the fundamental parallelogram, 
and is given by 
$$
\zeta=\frac{1-\epsilon}{2}\tau+\frac{1-\epsilon^{\prime}}{2}. 
$$

\subsection{Jacobi's triple product identity}
All the theta functions have infinite product expansions given by 
\begin{align}
\theta 
\left[
\begin{array}{c}
\epsilon \\
\epsilon^{\prime}
\end{array}
\right] (\zeta, \tau) &=\exp\left(\frac{\pi i \epsilon \epsilon^{\prime}}{2}\right) x^{\frac{\epsilon^2}{4}} z^{\frac{\epsilon}{2}}    \notag  \\
                           &\quad 
                           \displaystyle \times\prod_{n=1}^{\infty}(1-x^{2n})(1+e^{\pi i \epsilon^{\prime}} x^{2n-1+\epsilon} z)(1+e^{-\pi i \epsilon^{\prime}} x^{2n-1-\epsilon}/z),  \label{eqn:Jacobi-triple}
\end{align}
where $x=\exp(\pi i \tau)$ and $z=\exp(2\pi i \zeta).$ 
Therefore, according to Jacobi's derivative formula (\ref{eqn:Jacobi-derivative}), 
it follows that 
\begin{equation*}
\label{eqn:Jacobi}
\theta^{\prime} 
\left[
\begin{array}{c}
1 \\
1
\end{array}
\right](0,\tau) 
=
-2\pi 
q^{\frac18}
\prod_{n=1}^{\infty}(1-q^n)^3, \,\,q=\exp(2\pi i \tau). 
\end{equation*}

\subsection{Spaces of $N$-th order $\theta$-functions}

Following the theory by Farkas and Kra \cite{Farkas-Kra}, 
we define 
$\mathcal{F}_{N}\left[
\begin{array}{c}
\epsilon \\
\epsilon^{\prime}
\end{array}
\right] $ as the set of all functions $f$ satisfying the two functional equations, 
$$
f(\zeta+1)=\exp(\pi i \epsilon) \,\,f(\zeta),
$$
and 
$$
f(\zeta+\tau)=\exp(-\pi i)[\epsilon^{\prime}+2N\zeta+N\tau] \,\,f(\zeta), \quad \zeta\in\mathbb{C},  \,\,\tau \in\mathbb{H}^2,
$$ 
where 
$N$ is a positive integer and 
$\left[
\begin{array}{c}
\epsilon \\
\epsilon^{\prime}
\end{array}
\right] \in\mathbb{R}^2.$ 
This set of functions is called the space of {\it $N$-th order $\theta$-functions with characteristics }
$\left[
\begin{array}{c}
\epsilon \\
\epsilon^{\prime}
\end{array}
\right]. $ 
Note that 
$$
\dim \mathcal{F}_{N}\left[
\begin{array}{c}
\epsilon \\
\epsilon^{\prime}
\end{array}
\right] =N.
$$
For the proof, please refer to \cite[pp.133]{Farkas-Kra}.

\subsection{The heat equation}
The theta function satisfies the following heat equation: 
\begin{equation}
\label{eqn:heat}
\frac{\partial^2}{\partial \zeta^2}
\theta
\left[
\begin{array}{c}
\epsilon \\
\epsilon^{\prime}
\end{array}
\right](\zeta,\tau)
=
4\pi i
\frac{\partial}{\partial \tau}
\theta
\left[
\begin{array}{c}
\epsilon \\
\epsilon^{\prime}
\end{array}
\right](\zeta,\tau). 
\end{equation}

\section{Weierstrass elliptic function theory  }
\label{sec:Weierstrass}

In this paper, we introduce Weierstrass $\wp$-function and $\sigma$-function as follows:
\begin{align*}
\wp(z;\omega_1, \omega_2)=\wp(z)
=&
\frac{1}{z^2}+\sum_
{
\tiny{
\begin{matrix}
(m,n)\in\mathbb{Z}^2 \\ 
(m,n)\neq(0,0)
\end{matrix}
}
}
\left(
\frac{1}{(z-m\omega_1-n\omega_2)^2}-\frac{1}{(m\omega_1+n\omega_2)^2}
\right),   \\
\sigma(z;\omega_1, \omega_2)=\sigma(z)
=&
z
\prod_
{
\tiny{
\begin{matrix}
(m,n)\in\mathbb{Z}^2 \\ 
(m,n)\neq(0,0)
\end{matrix}
}
}\left( 1-\frac{z}{m\omega_1+n\omega_2}  \right)
\exp
\left(
\frac{z}{m\omega_1+n\omega_2}
+
\frac{z^2}{2(m\omega_1+n\omega_2)^2}
\right),
\end{align*}
where $z,\omega_1,\omega_2$ are complex numbers with $\omega_2/ \omega_1 \not\in\mathbb{R}.$
\par
The following formula is recalled from the study by Whittaker and Watson \cite[pp. 437, 459]{WW}:
\begin{align*}
\wp^{\prime}(z)^2=&4\wp(z)^3-g_2\wp(z)-g_3,  \\
g_2(\omega_1, \omega_2)=&
60G_4(\omega_1, \omega_2)
=
60
\sum_
{
\tiny{
\begin{matrix}
(m,n)\in\mathbb{Z}^2 \\ 
(m,n)\neq(0,0)
\end{matrix}
}
}
\frac{1}{(m\omega_1+n\omega_2)^4},  \\
g_3(\omega_1, \omega_2)=&
140G_6(\omega_1, \omega_2)
=
140
\sum_
{
\tiny{
\begin{matrix}
(m,n)\in\mathbb{Z}^2 \\ 
(m,n)\neq(0,0)
\end{matrix}
}
}
\frac{1}{(m\omega_1+n\omega_2)^6}, 
\end{align*}
which implies 
\begin{equation}
\wp^{\prime \prime}(z)=6\wp^2(z)-\frac12g_2. 
\end{equation}
For $k\in\mathbb{N},$ the Eisenstein series is defined by 
\begin{equation*}
G_{2k}(\tau)=G_{2k}(1, \tau)
=\sum_
{
\tiny{
\begin{matrix}
(m,n)\in\mathbb{Z}^2 \\ 
(m,n)\neq(0,0)
\end{matrix}
}
}
\frac{1}{(m+n\tau)^{2k}}, \,\,\tau\in\mathbb{H}^2. 
\end{equation*}
Moreover, from \cite[pp. 124]{Farkas-Kra}, we recall 
\begin{align*}
\wp(z;1, \tau)=&
\frac{1}{3}
\frac
{
\theta^{\prime  \prime \prime} 
\left[
\begin{array}{c}
1 \\
1
\end{array}
\right]
}
{
\theta^{\prime } 
\left[
\begin{array}{c}
1 \\
1
\end{array}
\right]
}-
\frac{d^2}{dz^2} \log 
\theta
\left[
\begin{array}{c}
1 \\
1
\end{array}
\right](z,\tau),  \\
\sigma(z;\omega_1, \omega_2)=&
\exp
\left(
\frac{\eta_1 z^2}{2\omega_1}
\right)
\frac{\omega_1}
{ 
\theta^{\prime}
\left[
\begin{array}{c}
1 \\
1
\end{array}
\right] }
\theta
\left[
\begin{array}{c}
1 \\
1
\end{array}
\right]
\left(
\frac{z}{\omega_1},
\tau
\right),
\end{align*}
where $\omega_2/\omega_1=\tau\in\mathbb{H}^2.$

\section{Preliminary results}
\label{sec:proof-preliminary}

\begin{proposition}
\label{prop:1st-derivative-(1,1/5)-(1,3/5)}
{\it
For every $\tau \in\mathbb{H}^2,$ we have 
\begin{align*}
\frac{
\theta^{\prime}
\left[
\begin{array}{c}
1 \\
\frac15
\end{array}
\right]  
}
{ 
\theta
\left[
\begin{array}{c}
1 \\
\frac15
\end{array}
\right]  
}=&
-\pi
\left\{
\cot \frac{2 \pi}{5} + 4 \sin \frac{4 \pi}{5} \sum_{n=1}^{\infty} (d_{1,5}(n)-d_{4,5}(n)) q^n
-4\sin \frac{2\pi}{5} \sum_{n=1}^{\infty} (d_{2,5}(n)-d_{3,5}(n)) q^n
\right\},  \\
\frac{
\theta^{\prime}
\left[
\begin{array}{c}
1 \\
\frac35
\end{array}
\right]  
}
{ 
\theta
\left[
\begin{array}{c}
1 \\
\frac35
\end{array}
\right]  
}=&
-\pi
\left\{
\cot \frac{ \pi}{5} + 4 \sin \frac{2 \pi}{5} \sum_{n=1}^{\infty} (d_{1,5}(n)-d_{4,5}(n)) q^n
+4\sin \frac{4\pi}{5} \sum_{n=1}^{\infty} (d_{2,5}(n)-d_{3,5}(n)) q^n
\right\},
\end{align*}
where $q=\exp(2\pi i \tau).$ 
}
\end{proposition}

\begin{proof}
The proposition follows from Jacobi's triple product identity (\ref{eqn:Jacobi-triple}). 
\end{proof}

\begin{proposition}
\label{prop:1st-derivative-(1/5,1)-(3/5,1)}
{\it
For every $\tau \in\mathbb{H}^2,$ we have 
\begin{align*}
\frac{
\theta^{\prime}
\left[
\begin{array}{c}
\frac15 \\
1
\end{array}
\right]  
}
{ 
\theta
\left[
\begin{array}{c}
\frac15 \\
1
\end{array}
\right]  
}=&
\frac{\pi i}{5}
\left\{
1 + 10 \sum_{n=1}^{\infty} (d_{2,5}(n)-d_{3,5}(n)) y^n
\right\},  \,\,
\frac{
\theta^{\prime}
\left[
\begin{array}{c}
\frac35 \\
1
\end{array}
\right]  
}
{ 
\theta
\left[
\begin{array}{c}
\frac35 \\
1
\end{array}
\right]  
}=&
\frac{\pi i}{5}
\left\{
3 + 10 \sum_{n=1}^{\infty} (d_{1,5}(n)-d_{4,5}(n)) y^n
\right\},
\end{align*}
where $y=\exp(2\pi i \tau/5).$ 
}
\end{proposition}

\begin{proof}
The proposition follows from Jacobi's triple product identity (\ref{eqn:Jacobi-triple}). 
\end{proof}

\begin{proposition}
\label{prop:E2}
{\it
For every $\tau\in \mathbb{H}^2,$ we have 
\begin{equation*}
\frac
{
\theta^{\prime\prime \prime}
\left[
\begin{array}{c}
1 \\
1
\end{array}
\right]
}
{
\theta^{\prime}
\left[
\begin{array}{c}
1 \\
1
\end{array}
\right]
}
=4 \pi i \frac{d}{d\tau} \log \theta^{\prime}
\left[
\begin{array}{c}
1 \\
1
\end{array}
\right]
=
-\pi^2E_2(q), \,\,q=\exp(2\pi i \tau). 
\end{equation*}
}
\end{proposition}

\begin{proof}
The proposition follows from Jacobi's triple product identity (\ref{eqn:Jacobi-triple}). 
\end{proof}

\section{Applications of the residue theorem}
\label{sec:application-residue-theorem}

\begin{theorem}
\label{thm-fundamental-elliptic-function}
{\it
The sum of all the residues of an elliptic function in the fundamental parallelogram is zero. 
}
\end{theorem}

\subsection{Derivative formulas of level five}

\begin{theorem}
\label{thm:derivative-level-5-(1,1/5)-(1,3/5)}
{\it
For every $\tau\in\mathbb{H}^2,$ we have 
\begin{align}
&
\frac{
\theta^{\prime}
\left[
\begin{array}{c}
1 \\
\frac15
\end{array}
\right]
}
{
\theta
\left[
\begin{array}{c}
1 \\
\frac15
\end{array}
\right]
}
=
\theta^{\prime}
\left[
\begin{array}{c}
1 \\
1
\end{array}
\right]
\frac
{
\left(
\theta^5
\left[
\begin{array}{c}
1 \\
\frac15
\end{array}
\right]
-
3
\theta^5
\left[
\begin{array}{c}
1 \\
\frac35
\end{array}
\right]
\right)
}
{
10
\theta^3
\left[
\begin{array}{c}
1 \\
\frac15
\end{array}
\right]
\theta^3
\left[
\begin{array}{c}
1 \\
\frac35
\end{array}
\right]
}, \label{eqn:analogue-Jacobi-(1,1/5)}  \\
&
\frac
{
\theta^{\prime}
\left[
\begin{array}{c}
1 \\
\frac35
\end{array}
\right]
}
{
\theta
\left[
\begin{array}{c}
1 \\
\frac35
\end{array}
\right]
}
=
\theta^{\prime}
\left[
\begin{array}{c}
1 \\
1
\end{array}
\right]
\frac
{
\left(
3
\theta^5
\left[
\begin{array}{c}
1 \\
\frac15
\end{array}
\right]
+
\theta^5
\left[
\begin{array}{c}
1 \\
\frac35
\end{array}
\right]
\right)
}
{
10
\theta^3
\left[
\begin{array}{c}
1 \\
\frac15
\end{array}
\right]
\theta^3
\left[
\begin{array}{c}
1 \\
\frac35
\end{array}
\right]
}.   \label{eqn:analogue-Jacobi-(1,3/5)}
\end{align}
}
\end{theorem}

\begin{proof}
Consider the following elliptic functions:
\begin{equation*}
\varphi(z)
=
\frac
{
\theta^3
\left[
\begin{array}{c}
1 \\
1
\end{array}
\right](z, \tau)
}
{
\theta^2
\left[
\begin{array}{c}
1 \\
\frac15
\end{array}
\right](z, \tau)
\theta
\left[
\begin{array}{c}
1 \\
\frac35
\end{array}
\right](z, \tau)
}
\,\,
\mathrm{and}
\,\,
\psi(z)=
\frac
{
\theta^3
\left[
\begin{array}{c}
1 \\
1
\end{array}
\right](z, \tau)
}
{
\theta^2
\left[
\begin{array}{c}
1 \\
\frac35
\end{array}
\right](z, \tau)
\theta
\left[
\begin{array}{c}
1 \\
-\frac15
\end{array}
\right](z, \tau)
}.
\end{equation*}
\par
Note that,
in the fundamental parallelogram,
the poles of $\varphi(z)$ are $z=2/5$ and $z=1/5.$
Thus, the direct calculation yields
\begin{equation*}
\mathrm{Res}\left(\varphi(z), \frac25\right)
=
\frac
{
\theta^3
\left[
\begin{array}{c}
1 \\
\frac15
\end{array}
\right]
}
{
\theta^{\prime}
\left[
\begin{array}{c}
1 \\
1
\end{array}
\right]^2
\theta
\left[
\begin{array}{c}
1 \\
\frac35
\end{array}
\right]
}
\left\{
-3
\frac
{
\theta^{\prime}
\left[
\begin{array}{c}
1 \\
\frac15
\end{array}
\right]
}
{
\theta
\left[
\begin{array}{c}
1 \\
\frac15
\end{array}
\right]
}
+
\frac
{
\theta^{\prime}
\left[
\begin{array}{c}
1 \\
\frac35
\end{array}
\right]
}
{
\theta
\left[
\begin{array}{c}
1 \\
\frac35
\end{array}
\right]
}
\right\}
\end{equation*}
and
\begin{equation*}
\mathrm{Res}\left(\varphi(z), \frac15  \right)
=
-
\frac
{
\theta
\left[
\begin{array}{c}
1 \\
\frac35
\end{array}
\right]
}
{
\theta^{\prime}
\left[
\begin{array}{c}
1 \\
1
\end{array}
\right]
}.
\end{equation*}
Because
$
\mathrm{Res}\left(\varphi(z), 2/5\right)+\mathrm{Res}\left(\varphi(z),1/5 \right)=0,
$
it is understood that
\begin{equation*}
3
\frac
{
\theta^{\prime}
\left[
\begin{array}{c}
1 \\
\frac15
\end{array}
\right]
}
{
\theta
\left[
\begin{array}{c}
1 \\
\frac15
\end{array}
\right]
}
-
\frac
{
\theta^{\prime}
\left[
\begin{array}{c}
1 \\
\frac35
\end{array}
\right]
}
{
\theta
\left[
\begin{array}{c}
1 \\
\frac35
\end{array}
\right]
}
=
-
\frac
{
\theta^{\prime}
\left[
\begin{array}{c}
1 \\
1
\end{array}
\right]
\theta^2
\left[
\begin{array}{c}
1 \\
\frac35
\end{array}
\right]
}
{
\theta^3
\left[
\begin{array}{c}
1 \\
\frac15
\end{array}
\right]
}.
\end{equation*}
\par
From $\psi(z),$ we have
\begin{equation*}
\frac
{
\theta^{\prime}
\left[
\begin{array}{c}
1 \\
\frac15
\end{array}
\right]
}
{
\theta
\left[
\begin{array}{c}
1 \\
\frac15
\end{array}
\right]
}
+3
\frac
{
\theta^{\prime}
\left[
\begin{array}{c}
1 \\
\frac35
\end{array}
\right]
}
{
\theta
\left[
\begin{array}{c}
1 \\
\frac35
\end{array}
\right]
}
=
\frac
{
\theta^{\prime}
\left[
\begin{array}{c}
1 \\
1
\end{array}
\right]
\theta^2
\left[
\begin{array}{c}
1 \\
\frac15
\end{array}
\right]
}
{
\theta^3
\left[
\begin{array}{c}
1 \\
\frac35
\end{array}
\right]
},
\end{equation*}
which proves the theorem.
\end{proof}

\begin{theorem}
\label{thm:derivative-level-5-(1/5,1)-(3/5,1)}
{\it
For every $\tau\in\mathbb{H}^2,$ we have 
\begin{align}
&
\frac{
\theta^{\prime}
\left[
\begin{array}{c}
\frac15 \\
1
\end{array}
\right]
}
{
\theta
\left[
\begin{array}{c}
\frac15 \\
1
\end{array}
\right]
}
=
-
\zeta_5^3
\theta^{\prime}
\left[
\begin{array}{c}
1 \\
1
\end{array}
\right]
\frac
{
\left(
\theta^5
\left[
\begin{array}{c}
\frac15 \\
1
\end{array}
\right]
+
3
\theta^5
\left[
\begin{array}{c}
\frac35 \\
1
\end{array}
\right]
\right)
}
{
10
\theta^3
\left[
\begin{array}{c}
\frac15 \\
1
\end{array}
\right]
\theta^3
\left[
\begin{array}{c}
\frac35 \\
1
\end{array}
\right]
},   \label{eqn:analogue-Jacobi-(1/5,1)}  \\
&
\frac
{
\theta^{\prime}
\left[
\begin{array}{c}
\frac35 \\
1
\end{array}
\right]
}
{
\theta
\left[
\begin{array}{c}
\frac35 \\
1
\end{array}
\right]
}
=
-
\zeta_5^3
\theta^{\prime}
\left[
\begin{array}{c}
1 \\
1
\end{array}
\right]
\frac
{
\left(
3
\theta^5
\left[
\begin{array}{c}
\frac15 \\
1
\end{array}
\right]
-
\theta^5
\left[
\begin{array}{c}
\frac35 \\
1
\end{array}
\right]
\right)
}
{
10
\theta^3
\left[
\begin{array}{c}
\frac15 \\
1
\end{array}
\right]
\theta^3
\left[
\begin{array}{c}
\frac35 \\
1
\end{array}
\right]
}. \label{eqn:analogue-Jacobi-(3/5,1)}
\end{align}
}
\end{theorem}

\begin{proof}
Consider the following elliptic functions:
\begin{equation*}
\varphi(z)
=
\frac
{
\theta^3
\left[
\begin{array}{c}
1 \\
1
\end{array}
\right](z, \tau)
}
{
\theta^2
\left[
\begin{array}{c}
\frac15 \\
1
\end{array}
\right](z, \tau)
\theta
\left[
\begin{array}{c}
\frac35 \\
1
\end{array}
\right](z, \tau)
}
\,\,
\mathrm{and}
\,\,
\psi(z)=
\frac
{
\theta^3
\left[
\begin{array}{c}
1 \\
1
\end{array}
\right](z, \tau)
}
{
\theta^2
\left[
\begin{array}{c}
\frac35 \\
1
\end{array}
\right](z, \tau)
\theta
\left[
\begin{array}{c}
-\frac15 \\
1
\end{array}
\right](z, \tau)
}.
\end{equation*}
Then, from $\varphi(z)$ and $\psi(z),$ we have
\begin{equation}
\label{eqn:relation-(1/5,1)-(3/5,1)-(1)}
3
\frac
{
\theta^{\prime}
\left[
\begin{array}{c}
\frac15 \\
1
\end{array}
\right]
}
{
\theta
\left[
\begin{array}{c}
\frac15 \\
1
\end{array}
\right]
}
-
\frac
{
\theta^{\prime}
\left[
\begin{array}{c}
\frac35 \\
1
\end{array}
\right]
}
{
\theta
\left[
\begin{array}{c}
\frac35 \\
1
\end{array}
\right]
}
=
-\zeta_5^3
\frac
{
\theta^{\prime}
\left[
\begin{array}{c}
1 \\
1
\end{array}
\right]
\theta^2
\left[
\begin{array}{c}
\frac35 \\
1
\end{array}
\right]
}
{
\theta^3
\left[
\begin{array}{c}
\frac15 \\
1
\end{array}
\right]
}
\end{equation}
and
\begin{equation}
\label{eqn:relation-(1/5,1)-(3/5,1)-(2)}
\frac
{
\theta^{\prime}
\left[
\begin{array}{c}
\frac15 \\
1
\end{array}
\right]
}
{
\theta
\left[
\begin{array}{c}
\frac15 \\
1
\end{array}
\right]
}
+3
\frac
{
\theta^{\prime}
\left[
\begin{array}{c}
\frac35 \\
1
\end{array}
\right]
}
{
\theta
\left[
\begin{array}{c}
\frac35 \\
1
\end{array}
\right]
}
=
-\zeta_5^3
\frac
{
\theta^{\prime}
\left[
\begin{array}{c}
1 \\
1
\end{array}
\right]
\theta^2
\left[
\begin{array}{c}
\frac15 \\
1
\end{array}
\right]
}
{
\theta^3
\left[
\begin{array}{c}
\frac35 \\
1
\end{array}
\right]
}.
\end{equation}
Thus, the theorem is proved.
\end{proof}

\subsection{Second-order theta derivatives}

\begin{proposition}
\label{prop:2nd-derivative-(1,1/5)-(1,3/5)}
{\it
For every $\tau\in\mathbb{H}^2,$ we have 
\begin{equation}
\label{eqn:2nd-deri-(1,1/5)-(1,3/5)-(1)}
2
\frac
{
\theta^{\prime \prime}
\left[
\begin{array}{c}
1 \\
\frac15
\end{array}
\right]
}
{
\theta
\left[
\begin{array}{c}
1 \\
\frac15
\end{array}
\right]
}
+
\frac
{
\theta^{\prime \prime}
\left[
\begin{array}{c}
1 \\
\frac35
\end{array}
\right]
}
{
\theta
\left[
\begin{array}{c}
1 \\
\frac35
\end{array}
\right]
}
-
\frac
{
\theta^{\prime\prime \prime}
\left[
\begin{array}{c}
1 \\
1
\end{array}
\right]
}
{
\theta^{\prime}
\left[
\begin{array}{c}
1 \\
1
\end{array}
\right]
}
+
2
\left\{
\frac
{
\theta^{\prime}
\left[
\begin{array}{c}
1 \\
\frac15
\end{array}
\right]
}
{
\theta
\left[
\begin{array}{c}
1 \\
\frac15
\end{array}
\right]
}
\right\}^2
+
4
\frac
{
\theta^{\prime}
\left[
\begin{array}{c}
1 \\
\frac15
\end{array}
\right]
}
{
\theta
\left[
\begin{array}{c}
1 \\
\frac15
\end{array}
\right]
}
\cdot
\frac
{
\theta^{\prime}
\left[
\begin{array}{c}
1 \\
\frac35
\end{array}
\right]
}
{
\theta
\left[
\begin{array}{c}
1 \\
\frac35
\end{array}
\right]
}
=0,
\end{equation}
\begin{equation}
\label{eqn:2nd-deri-(1,1/5)-(1,3/5)-(2)}
\frac
{
\theta^{\prime \prime}
\left[
\begin{array}{c}
1 \\
\frac15
\end{array}
\right]
}
{
\theta
\left[
\begin{array}{c}
1 \\
\frac15
\end{array}
\right]
}
+
2
\frac
{
\theta^{\prime \prime}
\left[
\begin{array}{c}
1 \\
\frac35
\end{array}
\right]
}
{
\theta
\left[
\begin{array}{c}
1 \\
\frac35
\end{array}
\right]
}
-
\frac
{
\theta^{\prime\prime \prime}
\left[
\begin{array}{c}
1 \\
1
\end{array}
\right]
}
{
\theta^{\prime}
\left[
\begin{array}{c}
1 \\
1
\end{array}
\right]
}
-4
\frac
{
\theta^{\prime}
\left[
\begin{array}{c}
1 \\
\frac15
\end{array}
\right]
}
{
\theta
\left[
\begin{array}{c}
1 \\
\frac15
\end{array}
\right]
}
\cdot
\frac
{
\theta^{\prime}
\left[
\begin{array}{c}
1 \\
\frac35
\end{array}
\right]
}
{
\theta
\left[
\begin{array}{c}
1 \\
\frac35
\end{array}
\right]
}
+
2
\left\{
\frac
{
\theta^{\prime}
\left[
\begin{array}{c}
1 \\
\frac35
\end{array}
\right]
}
{
\theta
\left[
\begin{array}{c}
1 \\
\frac35
\end{array}
\right]
}
\right\}^2
=0.
\end{equation}
}
\end{proposition}

\begin{proof}
Consider the following elliptic functions: 
\begin{equation*}
\frac
{
\theta^2
\left[
\begin{array}{c}
1 \\
\frac15
\end{array}
\right]
(z,\tau)
\theta
\left[
\begin{array}{c}
1 \\
\frac35
\end{array}
\right]
(z,\tau)
}
{
\theta^3
\left[
\begin{array}{c}
1 \\
1
\end{array}
\right]
(z,\tau)
}, \quad
\frac
{
\theta
\left[
\begin{array}{c}
1 \\
-\frac15
\end{array}
\right]
(z,\tau)
\theta^2
\left[
\begin{array}{c}
1 \\
\frac35
\end{array}
\right]
(z,\tau)
}
{
\theta^3
\left[
\begin{array}{c}
1 \\
1
\end{array}
\right]
(z,\tau)
}.
\end{equation*}
The proposition can be proved using Theorem \ref{thm-fundamental-elliptic-function}. 
\end{proof}

\begin{proposition}
\label{prop:2nd-derivative-(1/5,1)-(3/5,1)}
{\it
For every $\tau\in\mathbb{H}^2,$ we have 
\begin{equation}
\label{eqn:2nd-deri-(1/5,1)-(3/5,1)-(1)}
2
\frac
{
\theta^{\prime \prime}
\left[
\begin{array}{c}
\frac15 \\
1
\end{array}
\right]
}
{
\theta
\left[
\begin{array}{c}
\frac15 \\
1
\end{array}
\right]
}
+
\frac
{
\theta^{\prime \prime}
\left[
\begin{array}{c}
\frac35 \\
1
\end{array}
\right]
}
{
\theta
\left[
\begin{array}{c}
\frac35\\
1
\end{array}
\right]
}
-
\frac
{
\theta^{\prime\prime \prime}
\left[
\begin{array}{c}
1 \\
1
\end{array}
\right]
}
{
\theta^{\prime}
\left[
\begin{array}{c}
1 \\
1
\end{array}
\right]
}
+
2
\left\{
\frac
{
\theta^{\prime}
\left[
\begin{array}{c}
\frac15 \\
1
\end{array}
\right]
}
{
\theta
\left[
\begin{array}{c}
\frac15 \\
1
\end{array}
\right]
}
\right\}^2
+
4
\frac
{
\theta^{\prime}
\left[
\begin{array}{c}
\frac15 \\
1
\end{array}
\right]
}
{
\theta
\left[
\begin{array}{c}
\frac15 \\
1
\end{array}
\right]
}
\cdot
\frac
{
\theta^{\prime}
\left[
\begin{array}{c}
\frac35 \\
1
\end{array}
\right]
}
{
\theta
\left[
\begin{array}{c}
\frac35\\
1
\end{array}
\right]
}
=0,
\end{equation}
\begin{equation}
\label{eqn:2nd-deri-(1/5,1)-(3/5,1)-(2)}
\frac
{
\theta^{\prime \prime}
\left[
\begin{array}{c}
\frac15\\
1
\end{array}
\right]
}
{
\theta
\left[
\begin{array}{c}
\frac15 \\
1
\end{array}
\right]
}
+
2
\frac
{
\theta^{\prime \prime}
\left[
\begin{array}{c}
\frac35 \\
1
\end{array}
\right]
}
{
\theta
\left[
\begin{array}{c}
\frac35\\
1
\end{array}
\right]
}
-
\frac
{
\theta^{\prime\prime \prime}
\left[
\begin{array}{c}
1 \\
1
\end{array}
\right]
}
{
\theta^{\prime}
\left[
\begin{array}{c}
1 \\
1
\end{array}
\right]
}
-4
\frac
{
\theta^{\prime}
\left[
\begin{array}{c}
\frac15 \\
1
\end{array}
\right]
}
{
\theta
\left[
\begin{array}{c}
\frac15 \\
1
\end{array}
\right]
}
\cdot
\frac
{
\theta^{\prime}
\left[
\begin{array}{c}
\frac35 \\
1
\end{array}
\right]
}
{
\theta
\left[
\begin{array}{c}
\frac35\\
1
\end{array}
\right]
}
+
2
\left\{
\frac
{
\theta^{\prime}
\left[
\begin{array}{c}
\frac35 \\
1
\end{array}
\right]
}
{
\theta
\left[
\begin{array}{c}
\frac35 \\
1
\end{array}
\right]
}
\right\}^2
=0.
\end{equation}
}
\end{proposition}

\begin{proof}
Consider the following elliptic functions: 
\begin{equation*}
\frac
{
\theta^2
\left[
\begin{array}{c}
\frac15 \\
1
\end{array}
\right]
(z,\tau)
\theta
\left[
\begin{array}{c}
\frac35 \\
1
\end{array}
\right]
(z,\tau)
}
{
\theta^3
\left[
\begin{array}{c}
1 \\
1
\end{array}
\right]
(z,\tau)
}, \quad
\frac
{
\theta
\left[
\begin{array}{c}
-\frac15 \\
1
\end{array}
\right]
(z,\tau)
\theta^2
\left[
\begin{array}{c}
\frac35 \\
1
\end{array}
\right]
(z,\tau)
}
{
\theta^3
\left[
\begin{array}{c}
1 \\
1
\end{array}
\right]
(z,\tau)
}.
\end{equation*}
The proposition can be proved using Theorem \ref{thm-fundamental-elliptic-function}. 
\end{proof}

\subsection{Third-order theta derivatives}

\begin{proposition}
\label{prop:3rd-derivative-(1,1/5)-(1,3/5)}
{\it
For every $\tau\in\mathbb{H}^2,$ we have 
\begin{align*}
&
\frac
{
\theta^{(3)}
\left[
\begin{array}{c}
1 \\
\frac{1}{5}
\end{array}
\right]
}
{
\theta
\left[
\begin{array}{c}
1 \\
\frac{1}{5}
\end{array}
\right]
}
+
3
\frac
{
\theta^{(3)}
\left[
\begin{array}{c}
1 \\
\frac{3}{5}
\end{array}
\right]
}
{
\theta
\left[
\begin{array}{c}
1 \\
\frac{3}{5}
\end{array}
\right]
}
+
9
\frac
{
\theta^{(2)}
\left[
\begin{array}{c}
1 \\
\frac{1}{5}
\end{array}
\right]
}
{
\theta
\left[
\begin{array}{c}
1 \\
\frac{1}{5}
\end{array}
\right]
}
\frac
{
\theta^{(1)}
\left[
\begin{array}{c}
1 \\
\frac{3}{5}
\end{array}
\right]
}
{
\theta
\left[
\begin{array}{c}
1 \\
\frac{3}{5}
\end{array}
\right]
}
+
18
\frac
{
\theta^{(1)}
\left[
\begin{array}{c}
1 \\
\frac{1}{5}
\end{array}
\right]
}
{
\theta
\left[
\begin{array}{c}
1 \\
\frac{1}{5}
\end{array}
\right]
}
\left\{
\frac
{
\theta^{(1)}
\left[
\begin{array}{c}
1 \\
\frac{3}{5}
\end{array}
\right]
}
{
\theta
\left[
\begin{array}{c}
1 \\
\frac{3}{5}
\end{array}
\right]
}
\right\}^2  
+
9
\frac
{
\theta^{(1)}
\left[
\begin{array}{c}
1 \\
\frac{1}{5}
\end{array}
\right]
}
{
\theta
\left[
\begin{array}{c}
1 \\
\frac{1}{5}
\end{array}
\right]
}
\frac
{
\theta^{(2)}
\left[
\begin{array}{c}
1 \\
\frac{3}{5}
\end{array}
\right]
}
{
\theta
\left[
\begin{array}{c}
1 \\
\frac{3}{5}
\end{array}
\right]
}   \\
&+
6
\left\{
\frac
{
\theta^{(1)}
\left[
\begin{array}{c}
1 \\
\frac{3}{5}
\end{array}
\right]
}
{
\theta
\left[
\begin{array}{c}
1 \\
\frac{3}{5}
\end{array}
\right]
}
\right\}^3
+
18
\frac
{
\theta^{(1)}
\left[
\begin{array}{c}
1 \\
\frac{3}{5}
\end{array}
\right]
}
{
\theta
\left[
\begin{array}{c}
1 \\
\frac{3}{5}
\end{array}
\right]
}
\frac
{
\theta^{(2)}
\left[
\begin{array}{c}
1 \\
\frac{3}{5}
\end{array}
\right]
}
{
\theta
\left[
\begin{array}{c}
1 \\
\frac{3}{5}
\end{array}
\right]
}   
-
4
\frac
{
\theta^{(3)}
\left[
\begin{array}{c}
1 \\
1
\end{array}
\right]
}
{
\theta^{(1)}
\left[
\begin{array}{c}
1 \\
1
\end{array}
\right]
} 
\left(
\frac
{
\theta^{(1)}
\left[
\begin{array}{c}
1 \\
\frac{1}{5}
\end{array}
\right]
}
{
\theta
\left[
\begin{array}{c}
1 \\
\frac{1}{5}
\end{array}
\right]
}
+
3
\frac
{
\theta^{(1)}
\left[
\begin{array}{c}
1 \\
\frac{3}{5}
\end{array}
\right]
}
{
\theta
\left[
\begin{array}{c}
1 \\
\frac{3}{5}
\end{array}
\right]
}
\right)  =0,
\end{align*}
and
\begin{align*}
&
3
\frac
{
\theta^{(3)}
\left[
\begin{array}{c}
1 \\
\frac{1}{5}
\end{array}
\right]
}
{
\theta
\left[
\begin{array}{c}
1 \\
\frac{1}{5}
\end{array}
\right]
}
-
\frac
{
\theta^{(3)}
\left[
\begin{array}{c}
1 \\
\frac{3}{5}
\end{array}
\right]
}
{
\theta
\left[
\begin{array}{c}
1 \\
\frac{3}{5}
\end{array}
\right]
}
+
9
\frac
{
\theta^{(1)}
\left[
\begin{array}{c}
1 \\
\frac{1}{5}
\end{array}
\right]
}
{
\theta
\left[
\begin{array}{c}
1 \\
\frac{1}{5}
\end{array}
\right]
}
\frac
{
\theta^{(2)}
\left[
\begin{array}{c}
1 \\
\frac{3}{5}
\end{array}
\right]
}
{
\theta
\left[
\begin{array}{c}
1 \\
\frac{3}{5}
\end{array}
\right]
}
-
18
\frac
{
\theta^{(1)}
\left[
\begin{array}{c}
1 \\
\frac{3}{5}
\end{array}
\right]
}
{
\theta
\left[
\begin{array}{c}
1 \\
\frac{3}{5}
\end{array}
\right]
}
\left\{
\frac
{
\theta^{(1)}
\left[
\begin{array}{c}
1 \\
\frac{1}{5}
\end{array}
\right]
}
{
\theta
\left[
\begin{array}{c}
1 \\
\frac{1}{5}
\end{array}
\right]
}
\right\}^2  
-
9
\frac
{
\theta^{(2)}
\left[
\begin{array}{c}
1 \\
\frac{1}{5}
\end{array}
\right]
}
{
\theta
\left[
\begin{array}{c}
1 \\
\frac{1}{5}
\end{array}
\right]
}
\frac
{
\theta^{(1)}
\left[
\begin{array}{c}
1 \\
\frac{3}{5}
\end{array}
\right]
}
{
\theta
\left[
\begin{array}{c}
1 \\
\frac{3}{5}
\end{array}
\right]
}  \\
&+
6
\left\{
\frac
{
\theta^{(1)}
\left[
\begin{array}{c}
1 \\
\frac{1}{5}
\end{array}
\right]
}
{
\theta
\left[
\begin{array}{c}
1 \\
\frac{1}{5}
\end{array}
\right]
}
\right\}^3
+
18
\frac
{
\theta^{(1)}
\left[
\begin{array}{c}
1 \\
\frac{1}{5}
\end{array}
\right]
}
{
\theta
\left[
\begin{array}{c}
1 \\
\frac{1}{5}
\end{array}
\right]
}
\frac
{
\theta^{(2)}
\left[
\begin{array}{c}
1 \\
\frac{1}{5}
\end{array}
\right]
}
{
\theta
\left[
\begin{array}{c}
1 \\
\frac{1}{5}
\end{array}
\right]
}  
-
4
\frac
{
\theta^{(3)}
\left[
\begin{array}{c}
1 \\
1
\end{array}
\right]
}
{
\theta^{(1)}
\left[
\begin{array}{c}
1 \\
1
\end{array}
\right]
} 
\left(
3
\frac
{
\theta^{(1)}
\left[
\begin{array}{c}
1 \\
\frac{1}{5}
\end{array}
\right]
}
{
\theta
\left[
\begin{array}{c}
1 \\
\frac{1}{5}
\end{array}
\right]
}
-
\frac
{
\theta^{(1)}
\left[
\begin{array}{c}
1 \\
\frac{3}{5}
\end{array}
\right]
}
{
\theta
\left[
\begin{array}{c}
1 \\
\frac{3}{5}
\end{array}
\right]
}
\right)  =0.
\end{align*}
}
\end{proposition}

\begin{proof}
Consider the following elliptic functions:
\begin{equation*}
\frac
{
\theta
\left[
\begin{array}{c}
1 \\
\frac{1}{5}
\end{array}
\right](z)
\theta^3
\left[
\begin{array}{c}
1 \\
\frac{3}{5}
\end{array}
\right](z)
}
{
\theta^4
\left[
\begin{array}{c}
1 \\
1
\end{array}
\right](z)
}, \quad
\frac
{
\theta^3
\left[
\begin{array}{c}
1 \\
\frac{1}{5}
\end{array}
\right](z)
\theta
\left[
\begin{array}{c}
1 \\
-\frac{3}{5}
\end{array}
\right](z)
}
{
\theta^4
\left[
\begin{array}{c}
1 \\
1
\end{array}
\right](z)
}.
\end{equation*}
\end{proof}

\begin{proposition}
\label{prop:3rd-derivative-(1/5,1)-(3/5,1)}
{\it
For every $\tau\in\mathbb{H}^2,$ we have 
\begin{align*}
&
\frac
{
\theta^{(3)}
\left[
\begin{array}{c}
\frac{1}{5} \\
1
\end{array}
\right]
}
{
\theta
\left[
\begin{array}{c}
\frac{1}{5} \\
1
\end{array}
\right]
}
+
3
\frac
{
\theta^{(3)}
\left[
\begin{array}{c}
\frac{3}{5} \\
1
\end{array}
\right]
}
{
\theta
\left[
\begin{array}{c}
\frac{3}{5} \\
1
\end{array}
\right]
}
+
9
\frac
{
\theta^{(2)}
\left[
\begin{array}{c}
\frac{1}{5} \\
1
\end{array}
\right]
}
{
\theta
\left[
\begin{array}{c}
\frac{1}{5}\\
1
\end{array}
\right]
}
\frac
{
\theta^{(1)}
\left[
\begin{array}{c}
\frac{3}{5}\\
1
\end{array}
\right]
}
{
\theta
\left[
\begin{array}{c}
\frac{3}{5}\\
1
\end{array}
\right]
}
+
18
\frac
{
\theta^{(1)}
\left[
\begin{array}{c}
\frac{1}{5}\\
1
\end{array}
\right]
}
{
\theta
\left[
\begin{array}{c}
\frac{1}{5} \\
1
\end{array}
\right]
}
\left\{
\frac
{
\theta^{(1)}
\left[
\begin{array}{c}
\frac{3}{5} \\
1
\end{array}
\right]
}
{
\theta
\left[
\begin{array}{c}
\frac{3}{5} \\
1
\end{array}
\right]
}
\right\}^2  
+
9
\frac
{
\theta^{(1)}
\left[
\begin{array}{c}
\frac{1}{5} \\
1
\end{array}
\right]
}
{
\theta
\left[
\begin{array}{c}
\frac{1}{5} \\
1
\end{array}
\right]
}
\frac
{
\theta^{(2)}
\left[
\begin{array}{c}
\frac{3}{5} \\
1
\end{array}
\right]
}
{
\theta
\left[
\begin{array}{c}
\frac{3}{5}\\
1
\end{array}
\right]
}   \\
&+
6
\left\{
\frac
{
\theta^{(1)}
\left[
\begin{array}{c}
\frac{3}{5} \\
1
\end{array}
\right]
}
{
\theta
\left[
\begin{array}{c}
\frac{3}{5} \\
1
\end{array}
\right]
}
\right\}^3
+
18
\frac
{
\theta^{(1)}
\left[
\begin{array}{c}
\frac{3}{5}\\
1
\end{array}
\right]
}
{
\theta
\left[
\begin{array}{c}
\frac{3}{5} \\
1
\end{array}
\right]
}
\frac
{
\theta^{(2)}
\left[
\begin{array}{c}
\frac{3}{5} \\
1
\end{array}
\right]
}
{
\theta
\left[
\begin{array}{c}
\frac{3}{5} \\
1
\end{array}
\right]
}   
-
4
\frac
{
\theta^{(3)}
\left[
\begin{array}{c}
1 \\
1
\end{array}
\right]
}
{
\theta^{(1)}
\left[
\begin{array}{c}
1 \\
1
\end{array}
\right]
} 
\left(
\frac
{
\theta^{(1)}
\left[
\begin{array}{c}
\frac{1}{5}\\
1
\end{array}
\right]
}
{
\theta
\left[
\begin{array}{c}
\frac{1}{5} \\
1
\end{array}
\right]
}
+
3
\frac
{
\theta^{(1)}
\left[
\begin{array}{c}
\frac{3}{5} \\
1
\end{array}
\right]
}
{
\theta
\left[
\begin{array}{c}
\frac{3}{5} \\
1
\end{array}
\right]
}
\right)  =0,
\end{align*}
and
\begin{align*}
&
3
\frac
{
\theta^{(3)}
\left[
\begin{array}{c}
\frac{1}{5} \\
1
\end{array}
\right]
}
{
\theta
\left[
\begin{array}{c}
\frac{1}{5}\\
1
\end{array}
\right]
}
-
\frac
{
\theta^{(3)}
\left[
\begin{array}{c}
\frac{3}{5} \\
1
\end{array}
\right]
}
{
\theta
\left[
\begin{array}{c}
\frac{3}{5} \\
1
\end{array}
\right]
}
+
9
\frac
{
\theta^{(1)}
\left[
\begin{array}{c}
\frac{1}{5} \\
1
\end{array}
\right]
}
{
\theta
\left[
\begin{array}{c}
\frac{1}{5} \\
1
\end{array}
\right]
}
\frac
{
\theta^{(2)}
\left[
\begin{array}{c}
\frac{3}{5} \\
1
\end{array}
\right]
}
{
\theta
\left[
\begin{array}{c}
\frac{3}{5} \\
1
\end{array}
\right]
}
-
18
\frac
{
\theta^{(1)}
\left[
\begin{array}{c}
\frac{3}{5} \\
1
\end{array}
\right]
}
{
\theta
\left[
\begin{array}{c}
\frac{3}{5} \\
1
\end{array}
\right]
}
\left\{
\frac
{
\theta^{(1)}
\left[
\begin{array}{c}
\frac{1}{5} \\
1
\end{array}
\right]
}
{
\theta
\left[
\begin{array}{c}
\frac{1}{5} \\
1
\end{array}
\right]
}
\right\}^2  
-
9
\frac
{
\theta^{(2)}
\left[
\begin{array}{c}
\frac{1}{5} \\
1
\end{array}
\right]
}
{
\theta
\left[
\begin{array}{c}
\frac{1}{5} \\
1
\end{array}
\right]
}
\frac
{
\theta^{(1)}
\left[
\begin{array}{c}
\frac{3}{5} \\
1
\end{array}
\right]
}
{
\theta
\left[
\begin{array}{c}
\frac{3}{5} \\
1
\end{array}
\right]
}  \\
&+
6
\left\{
\frac
{
\theta^{(1)}
\left[
\begin{array}{c}
\frac{1}{5}\\
1
\end{array}
\right]
}
{
\theta
\left[
\begin{array}{c}
\frac{1}{5} \\
1
\end{array}
\right]
}
\right\}^3
+
18
\frac
{
\theta^{(1)}
\left[
\begin{array}{c}
\frac{1}{5} \\
1
\end{array}
\right]
}
{
\theta
\left[
\begin{array}{c}
\frac{1}{5}\\
1
\end{array}
\right]
}
\frac
{
\theta^{(2)}
\left[
\begin{array}{c}
\frac{1}{5} \\
1
\end{array}
\right]
}
{
\theta
\left[
\begin{array}{c}
\frac{1}{5} \\
1
\end{array}
\right]
}  
-
4
\frac
{
\theta^{(3)}
\left[
\begin{array}{c}
1 \\
1
\end{array}
\right]
}
{
\theta^{(1)}
\left[
\begin{array}{c}
1 \\
1
\end{array}
\right]
} 
\left(
3
\frac
{
\theta^{(1)}
\left[
\begin{array}{c}
\frac{1}{5} \\
1
\end{array}
\right]
}
{
\theta
\left[
\begin{array}{c}
\frac{1}{5}\\
1
\end{array}
\right]
}
-
\frac
{
\theta^{(1)}
\left[
\begin{array}{c}
\frac{3}{5} \\
1
\end{array}
\right]
}
{
\theta
\left[
\begin{array}{c}
\frac{3}{5} \\
1
\end{array}
\right]
}
\right)  =0.
\end{align*}
}
\end{proposition}

\begin{proof}
Furthermore, consider the following elliptic functions:
\begin{equation*}
\frac
{
\theta
\left[
\begin{array}{c}
1 \\
\frac{1}{5}
\end{array}
\right](z)
\theta^3
\left[
\begin{array}{c}
1 \\
\frac{3}{5}
\end{array}
\right](z)
}
{
\theta^4
\left[
\begin{array}{c}
1 \\
1
\end{array}
\right](z)
}, \quad
\frac
{
\theta^3
\left[
\begin{array}{c}
1 \\
\frac{1}{5}
\end{array}
\right](z)
\theta
\left[
\begin{array}{c}
1 \\
-\frac{3}{5}
\end{array}
\right](z)
}
{
\theta^4
\left[
\begin{array}{c}
1 \\
1
\end{array}
\right](z)
}.
\end{equation*}
\end{proof}

\subsection{Fourth-order theta derivatives}

\begin{proposition}
\label{prop:4th-derivative-(1,1/5)-(1,3/5)}
{\it
For every $\tau\in\mathbb{H}^2$ and $j=1,3,$ we have 
\begin{align*}
&
\frac
{
\theta^{(4)}
\left[
\begin{array}{c}
1 \\
\frac{j}{5}
\end{array}
\right]
}
{
\theta
\left[
\begin{array}{c}
1 \\
\frac{j}{5}
\end{array}
\right]
}
+
16
\frac
{
\theta^{(1)}
\left[
\begin{array}{c}
1 \\
\frac{j}{5}
\end{array}
\right]
}
{
\theta
\left[
\begin{array}{c}
1 \\
\frac{j}{5}
\end{array}
\right]
}
\frac
{
\theta^{(3)}
\left[
\begin{array}{c}
1 \\
\frac{j}{5}
\end{array}
\right]
}
{
\theta
\left[
\begin{array}{c}
1 \\
\frac{j}{5}
\end{array}
\right]
}
+
12
\left\{
\frac
{
\theta^{(2)}
\left[
\begin{array}{c}
1 \\
\frac{j}{5}
\end{array}
\right]
}
{
\theta
\left[
\begin{array}{c}
1 \\
\frac{j}{5}
\end{array}
\right]
}
\right\}^2  
+
72
\left\{
\frac
{
\theta^{(1)}
\left[
\begin{array}{c}
1 \\
\frac{j}{5}
\end{array}
\right]
}
{
\theta
\left[
\begin{array}{c}
1 \\
\frac{j}{5}
\end{array}
\right]
}
\right\}^2
\frac
{
\theta^{(2)}
\left[
\begin{array}{c}
1 \\
\frac{j}{5}
\end{array}
\right]
}
{
\theta
\left[
\begin{array}{c}
1 \\
\frac{j}{5}
\end{array}
\right]
}  \\
+&
24
\left\{
\frac
{
\theta^{(1)}
\left[
\begin{array}{c}
1 \\
\frac{j}{5}
\end{array}
\right]
}
{
\theta
\left[
\begin{array}{c}
1 \\
\frac{j}{5}
\end{array}
\right]
}
\right\}^4  
-40
\frac
{
\theta^{(3)}
\left[
\begin{array}{c}
1 \\
1
\end{array}
\right]
}
{
\theta^{(1)}
\left[
\begin{array}{c}
1 \\
1
\end{array}
\right]
}
\left\{
\frac
{
\theta^{(1)}
\left[
\begin{array}{c}
1 \\
\frac{j}{5}
\end{array}
\right]
}
{
\theta
\left[
\begin{array}{c}
1 \\
\frac{j}{5}
\end{array}
\right]
}
\right\}^2
-
10
\frac
{
\theta^{(3)}
\left[
\begin{array}{c}
1 \\
1
\end{array}
\right]
}
{
\theta^{(1)}
\left[
\begin{array}{c}
1 \\
1
\end{array}
\right]
}
\frac
{
\theta^{(2)}
\left[
\begin{array}{c}
1 \\
\frac{j}{5}
\end{array}
\right]
}
{
\theta
\left[
\begin{array}{c}
1 \\
\frac{j}{5}
\end{array}
\right]
}   \\
&-\frac15
\frac
{
\theta^{(5)}
\left[
\begin{array}{c}
1 \\
1
\end{array}
\right]
}
{
\theta^{(1)}
\left[
\begin{array}{c}
1 \\
1
\end{array}
\right]
}
+
2
\left\{
\frac
{
\theta^{(3)}
\left[
\begin{array}{c}
1 \\
1
\end{array}
\right]
}
{
\theta^{(1)}
\left[
\begin{array}{c}
1 \\
1
\end{array}
\right]
}
\right\}^2=0.
\end{align*}
}
\end{proposition}

\begin{proof}
Consider the following elliptic functions:
$
f_j(z)
=
\theta^5
\left[
\begin{array}{c}
1 \\
\frac{j}{5}
\end{array}
\right]
(z,\tau)
/
\theta^5
\left[
\begin{array}{c}
1 \\
1
\end{array}
\right]
(z,\tau).
$ 
\end{proof}

\begin{proposition}
\label{prop:4th-derivative-(1/5,1)-(3/5,1)}
{\it
For every $\tau\in\mathbb{H}^2$ and $j=1,3,$ we have 
\begin{align*}
&
\frac
{
\theta^{(4)}
\left[
\begin{array}{c}
\frac{j}{5} \\
1
\end{array}
\right]
}
{
\theta
\left[
\begin{array}{c}
\frac{j}{5} \\
1
\end{array}
\right]
}
+
16
\frac
{
\theta^{(1)}
\left[
\begin{array}{c}
\frac{j}{5} \\
1
\end{array}
\right]
}
{
\theta
\left[
\begin{array}{c}
\frac{j}{5} \\
1
\end{array}
\right]
}
\frac
{
\theta^{(3)}
\left[
\begin{array}{c}
\frac{j}{5} \\
1
\end{array}
\right]
}
{
\theta
\left[
\begin{array}{c}
\frac{j}{5} \\
1
\end{array}
\right]
}
+
12
\left\{
\frac
{
\theta^{(2)}
\left[
\begin{array}{c}
\frac{j}{5} \\
1
\end{array}
\right]
}
{
\theta
\left[
\begin{array}{c}
\frac{j}{5} \\
1
\end{array}
\right]
}
\right\}^2  
+
72
\left\{
\frac
{
\theta^{(1)}
\left[
\begin{array}{c}
\frac{j}{5} \\
1
\end{array}
\right]
}
{
\theta
\left[
\begin{array}{c}
\frac{j}{5} \\
1
\end{array}
\right]
}
\right\}^2
\frac
{
\theta^{(2)}
\left[
\begin{array}{c}
\frac{j}{5} \\
1
\end{array}
\right]
}
{
\theta
\left[
\begin{array}{c}
\frac{j}{5} \\
1
\end{array}
\right]
}  \\
+&
24
\left\{
\frac
{
\theta^{(1)}
\left[
\begin{array}{c}
\frac{j}{5} \\
1
\end{array}
\right]
}
{
\theta
\left[
\begin{array}{c}
\frac{j}{5} \\
1
\end{array}
\right]
}
\right\}^4  
-40
\frac
{
\theta^{(3)}
\left[
\begin{array}{c}
1 \\
1
\end{array}
\right]
}
{
\theta^{(1)}
\left[
\begin{array}{c}
1 \\
1
\end{array}
\right]
}
\left\{
\frac
{
\theta^{(1)}
\left[
\begin{array}{c}
\frac{j}{5} \\
1
\end{array}
\right]
}
{
\theta
\left[
\begin{array}{c}
\frac{j}{5} \\
1
\end{array}
\right]
}
\right\}^2
-
10
\frac
{
\theta^{(3)}
\left[
\begin{array}{c}
1 \\
1
\end{array}
\right]
}
{
\theta^{(1)}
\left[
\begin{array}{c}
1 \\
1
\end{array}
\right]
}
\frac
{
\theta^{(2)}
\left[
\begin{array}{c}
\frac{j}{5} \\
1
\end{array}
\right]
}
{
\theta
\left[
\begin{array}{c}
\frac{j}{5} \\
1
\end{array}
\right]
}   \\
&-\frac15
\frac
{
\theta^{(5)}
\left[
\begin{array}{c}
1 \\
1
\end{array}
\right]
}
{
\theta^{(1)}
\left[
\begin{array}{c}
1 \\
1
\end{array}
\right]
}
+
2
\left\{
\frac
{
\theta^{(3)}
\left[
\begin{array}{c}
1 \\
1
\end{array}
\right]
}
{
\theta^{(1)}
\left[
\begin{array}{c}
1 \\
1
\end{array}
\right]
}
\right\}^2=0.
\end{align*}
}
\end{proposition}

\begin{proof}
Consider the following elliptic functions:
$
f_j(z)
=
\theta^5
\left[
\begin{array}{c}
\frac{j}{5} \\
1
\end{array}
\right]
(z,\tau)
/
\theta^5
\left[
\begin{array}{c}
1 \\
1
\end{array}
\right]
(z,\tau).
$ 
\end{proof}

\section{Proof of Theorem \ref{thm-level5-E_2(q)}}
\label{sec:proof:level5-E_2(q)}

\subsection{Notations}
For every $\tau\in\mathbb{H}^2,$ set $q=\exp(2\pi i \tau)$ and 
\begin{equation*}
X=
\frac
{
\theta^{\prime}
\left[
\begin{array}{c}
1 \\
\frac15
\end{array}
\right]
}
{
\theta
\left[
\begin{array}{c}
1 \\
\frac15
\end{array}
\right]
},  \,\,
Y=
\frac
{
\theta^{\prime}
\left[
\begin{array}{c}
1 \\
\frac35
\end{array}
\right]
}
{
\theta
\left[
\begin{array}{c}
1 \\
\frac35
\end{array}
\right]
},  \,\,
Z=
\frac
{
\theta^{\prime  \prime  \prime}
\left[
\begin{array}{c}
1 \\
1
\end{array}
\right]
}
{
\theta^{\prime  }
\left[
\begin{array}{c}
1 \\
1
\end{array}
\right]
}. 
\end{equation*}

\subsection{ODEs for $P(q)$ and $Q(q)$}

\begin{proof}
Based on Proposition \ref{prop:2nd-derivative-(1,1/5)-(1,3/5)}, we have 
\begin{equation}
\label{eqn:2nd-deri-X-Y-(1,1/5)-(1,3/5)}
\frac{
\theta^{\prime \prime}
\left[
\begin{array}{c}
1 \\
\frac15
\end{array}
\right]  
}
{ 
\theta
\left[
\begin{array}{c}
1 \\
\frac15
\end{array}
\right]
}
=
\frac13 Z
-4XY-\frac43X^2+\frac23Y^2,   \,\,
\frac{
\theta^{\prime \prime}
\left[
\begin{array}{c}
1 \\
\frac35
\end{array}
\right]  
}
{ 
\theta
\left[
\begin{array}{c}
1 \\
\frac35
\end{array}
\right]
}
=
\frac13 Z
+4XY+\frac23X^2-\frac43Y^2.  
\end{equation}
Based on Proposition \ref{prop:3rd-derivative-(1,1/5)-(1,3/5)}, we obtain
\begin{equation}
\label{eqn:3rd-deri-X-Y-(1,1/5)-(1,3/5)}
\frac{
\theta^{\prime \prime \prime}
\left[
\begin{array}{c}
1 \\
\frac15
\end{array}
\right]  
}
{ 
\theta
\left[
\begin{array}{c}
1 \\
\frac15
\end{array}
\right]
}
=
3X^3+9X^2Y-15XY^2+3Y^3+XZ,   \,\,
\frac{
\theta^{\prime \prime \prime}
\left[
\begin{array}{c}
1 \\
\frac35
\end{array}
\right]  
}
{ 
\theta
\left[
\begin{array}{c}
1 \\
\frac35
\end{array}
\right]
}
=
-3X^3-15X^2Y-9XY^2+3Y^3+YZ.  
\end{equation}
\par
Because 
\begin{equation*}
\frac{
\theta^{\prime \prime \prime}
\left[
\begin{array}{c}
1 \\
\frac{j}{5}
\end{array}
\right]  
}
{ 
\theta
\left[
\begin{array}{c}
1 \\
\frac{j}{5}
\end{array}
\right]
}
=
4 \pi i
\frac{d}{d\tau}
\left\{
\frac{
\theta^{\prime }
\left[
\begin{array}{c}
1 \\
\frac{j}{5}
\end{array}
\right]  
}
{ 
\theta
\left[
\begin{array}{c}
1 \\
\frac{j}{5}
\end{array}
\right]
}
\right\}
+
\frac{
\theta^{\prime }
\left[
\begin{array}{c}
1 \\
\frac{j}{5}
\end{array}
\right]  
}
{ 
\theta
\left[
\begin{array}{c}
1 \\
\frac{j}{5}
\end{array}
\right]
}
\cdot
\frac{
\theta^{\prime \prime }
\left[
\begin{array}{c}
1 \\
\frac{j}{5}
\end{array}
\right]  
}
{ 
\theta
\left[
\begin{array}{c}
1 \\
\frac{j}{5}
\end{array}
\right]
},  \,\,(j=1,3), 
\end{equation*}
it follows that 
\begin{align*}
4\pi i X^{\prime}=&
\frac{13}{3}X^3+13X^2Y-\frac{47}{3}XY^2+3Y^3+\frac23XZ,  \\
4\pi i Y^{\prime}=&
-3X^3-\frac{47}{3}X^2Y-13XY^2+\frac{13}{3}Y^3+\frac23YZ, 
\end{align*}
where ${ }^{\prime}=d/d\tau. $ 
By considering $d/d\tau=2 \pi i q d/dq,$ we obtain the ODEs for $P(q)$ and $Q(q)$. 
\end{proof}

\subsection{ODE for $R(q)$}

\begin{proof}
First, note that 
\begin{align*}
\frac{
\theta^{(4)}
\left[
\begin{array}{c}
1 \\
\frac15
\end{array}
\right]  
}
{ 
\theta
\left[
\begin{array}{c}
1 \\
\frac15
\end{array}
\right]
}
=&
4 \pi i
\frac{d}{d\tau}
\left\{
\frac{
\theta^{\prime \prime }
\left[
\begin{array}{c}
1 \\
\frac15
\end{array}
\right]  
}
{ 
\theta
\left[
\begin{array}{c}
1 \\
\frac15
\end{array}
\right]
}
\right\}
+
\left\{
\frac{
\theta^{\prime \prime}
\left[
\begin{array}{c}
1 \\
\frac15
\end{array}
\right]  
}
{ 
\theta
\left[
\begin{array}{c}
1 \\
\frac15
\end{array}
\right]
}
\right\}^2,   \\
=&
\frac{4\pi i}{3}\frac{dZ}{d\tau}
+\frac{20}{9}X^4+\frac{52}{3}X^3Y+\frac{316}{9}X^2Y^2+\frac{44}{3}XY^3-\frac{52}{9}Y^4  \\
&\quad -8XYZ-\frac83X^2Z+\frac43Y^2Z+\frac19Z^2,  
\end{align*}
and 
\begin{align*}
\frac{
\theta^{(5)}
\left[
\begin{array}{c}
1 \\
1
\end{array}
\right]  
}
{ 
\theta^{(1)}
\left[
\begin{array}{c}
1 \\
1
\end{array}
\right]
}
=&
4 \pi i
\frac{d}{d\tau}
\left\{
\frac{
\theta^{(3) }
\left[
\begin{array}{c}
1 \\
1
\end{array}
\right]  
}
{ 
\theta^{(1)}
\left[
\begin{array}{c}
1 \\
1
\end{array}
\right]
}
\right\}
+
\left\{
\frac{
\theta^{(3)}
\left[
\begin{array}{c}
1 \\
1
\end{array}
\right]  
}
{ 
\theta^{(1)}
\left[
\begin{array}{c}
1 \\
1
\end{array}
\right]
}
\right\}^2   
=
4\pi i \frac{dZ}{d\tau}+Z^2.
\end{align*}
By using Proposition \ref{prop:4th-derivative-(1,1/5)-(1,3/5)}, 
we have 
\begin{equation*}
4\pi i Z^{\prime}=
\frac{10}{3}X^4-10X^3Y-\frac{310}{3}X^2Y^2+10XY^3+\frac{10}{3}Y^4+\frac23 Z^2,
\end{equation*}
and
\begin{equation*}
\frac{
\theta^{(4)}
\left[
\begin{array}{c}
1 \\
\frac15
\end{array}
\right]  
}
{ 
\theta
\left[
\begin{array}{c}
1 \\
\frac15
\end{array}
\right]
}
=
\frac{10}{3}X^4+14X^3Y+\frac23 X^2Y^2+18XY^3-\frac{14}{3}XY^3+\frac13Z^2-8XYZ-\frac83 X^2Z+\frac43 Y^2Z, 
\end{equation*}
where ${ }^{\prime}=d/d\tau.$ By considering $d/d\tau=2 \pi i q d/dq,$ we obtain the ODE for $R(q)$. 
\end{proof}

\subsection{Note on $E_4$ and $E_6$}

\begin{theorem}
\label{thm-E4-E6-level5-q}
{\it
For $q\in\mathbb{C}$ with $|q|<1,$ 
set 
\begin{align*}
P(q)=&\cot \frac{2 \pi}{5} + 4 \sin \frac{4 \pi}{5} \sum_{n=1}^{\infty} (d_{1,5}(n)-d_{4,5}(n)) q^n
-4\sin \frac{2\pi}{5} \sum_{n=1}^{\infty} (d_{2,5}(n)-d_{3,5}(n)) q^n, \\
Q(q)=&\cot \frac{ \pi}{5} + 4 \sin \frac{2 \pi}{5} \sum_{n=1}^{\infty} (d_{1,5}(n)-d_{4,5}(n)) q^n
+4\sin \frac{4\pi}{5} \sum_{n=1}^{\infty} (d_{2,5}(n)-d_{3,5}(n)) q^n. 
\end{align*}
Then, we have 
\begin{align*}
E_4(q)=&
-5P^4+15P^3Q+155P^2Q^2-15PQ^3-5Q^4, \\
E_6(q)=&
\frac
{
-395P^6-3690P^5Q-8025P^4Q^2-8025P^2Q^4+3690PQ^5-395Q^6
}
{
8
}. 
\end{align*}
}
\end{theorem}

\begin{proof}
For the proof, we use 
\begin{equation*}
\wp=\wp(z;1,\tau)
=
\frac13
\frac
{
\theta^{\prime  \prime  \prime}
\left[
\begin{array}{c}
1 \\
1
\end{array}
\right]
}
{
\theta^{\prime  }
\left[
\begin{array}{c}
1 \\
1
\end{array}
\right]
}
-
\frac{d^2}{dz^2}
\log
\theta
\left[
\begin{array}{c}
1 \\
1
\end{array}
\right](z). 
\end{equation*} 
\par
Note that 
\begin{equation*}
\wp^{\prime \prime}=6\wp^2-g_2, \,\,
g_2(1,\tau)=\frac{4\pi^4}{3}E_4(q).
\end{equation*}
By substituting $z=-2/5,$ we have 
\begin{equation}
\label{eqn:g2-X-Y(1)}
g_2=-\frac{20}{3}X^4+20X^3Y+\frac{620}{3}X^2Y^2-20XY^3-\frac{20}{3}Y^4, 
\end{equation}
which implies 
\begin{equation*}
E_4(q)=-5P^4+15P^3Q+155P^2Q^2-15PQ^3-5Q^4. 
\end{equation*}
\par
Next, note that 
\begin{equation*}
(\wp^{\prime})^2=4\wp^3-g_2\wp-g_3, \,\,
g_3(1,\tau)=\frac{8\pi^6}{27}E_6(q).
\end{equation*}
By substituting $z=-2/5,$ we obtain 
\begin{equation}
\label{eqn:g3-X-Y(1)}
g_3=-\frac{395}{27}X^6-\frac{410}{3}X^5Y-\frac{2675}{9}X^4Y^2
-\frac{2675}{9}X^2Y^4+\frac{410}{3}XY^5-\frac{395}{27}Y^6, 
\end{equation}
which implies 
\begin{equation*}
E_6(q)=\frac
{
-395P^6-3690P^5Q-8025P^4Q^2-8025P^2Q^4+3690PQ^5-395Q^6
}
{
8
}. 
\end{equation*}
\end{proof}

\begin{corollary}
\label{coro-E4-E6-level5-q-(S,T)}
{\it
For $q\in\mathbb{C}$ with $|q|<1,$ 
set 
\begin{align*}
S(q)=&
\frac
{
\sqrt{250-110 \sqrt{5}}
}
{
5
}
\prod_{n=1}^{\infty}
\frac
{
(1-q^n)^5
}
{
(1-q^{5n})^3
}
\left(
1+
\frac{1-\sqrt{5}}{2} 
q^n
+
q^{2n} 
\right)^5  \\ 
=&
\frac
{
\sqrt{250-110 \sqrt{5}}
}
{
5
}
\left\{
1
-
\frac52(1+\sqrt{5})
\sum_{n=1}^{\infty} (d_{1,5}(n)-d_{4,5}(n)) q^n 
+
\frac52
(7+3 \sqrt{5})
\sum_{n=1}^{\infty} (d_{2,5}(n)-d_{3,5}(n)) q^n
\right\},  
\\
T(q)=&
\frac
{
\sqrt{250+110 \sqrt{5}}
}
{
5
}
\prod_{n=1}^{\infty}
\frac
{
(1-q^n)^5
}
{
(1-q^{5n})^3
}
\left(
1+
\frac{1+\sqrt{5}}{2} 
q^n
+
q^{2n} 
\right)^5  \\
=&
\frac
{
\sqrt{250+110 \sqrt{5}}
}
{
5
}
\left\{
1
+
\frac52(-1+\sqrt{5})
\sum_{n=1}^{\infty} (d_{1,5}(n)-d_{4,5}(n)) q^n 
+
\frac52
(7-3 \sqrt{5})
\sum_{n=1}^{\infty} (d_{2,5}(n)-d_{3,5}(n)) q^n
\right\} . 
\end{align*}
Then we have 
\begin{align*}
E_4(q)=&
\frac{S^4+12 S^3 T+14 S^2 T^2-12 S T^3+T^4}{16},  \\
E_6(q)=&
-\frac{1}{64} \left(S^2+T^2\right) \left(S^4+18 S^3 T+74 S^2 T^2-18 S T^3+T^4\right),
\end{align*}
and 
\begin{align*}
q(q:q)_{\infty}^{24}=&
\frac{
(E_4(q))^3-(E_6(q))^2
}{1728}
=
-\frac{S^5 T^5 \left(S^2+11 S T-T^2\right)}{4096}, \\
j(q)=&
-\frac{\left(S^4+12 S^3 T+14 S^2 T^2-12 S T^3+T^4\right)^3}{S^5 T^5 \left(S^2+11 S T-T^2\right)}
=
-\frac{\left(f^4-12 f^3+14 f^2+12 f+1\right)^3}{f^5 \left(-f^2+11 f+1\right)},   \,\,f=T/S.
\end{align*}
}
\end{corollary}

\begin{proof}
The corollary can be obtained by considering 
\begin{equation*}
S=-3 P+Q, \,\,
T= P+3Q, \,\,\mathrm{and} \,\,f=T/S.
\end{equation*}
\end{proof}

\begin{corollary}
\label{coro-E4-E6-level5-q-(V,W)}
{\it
For $q\in\mathbb{C}$ with $|q|<1,$ 
set 
\begin{align*}
V(q)=&\prod_{n=1}^{\infty}
\frac
{
(1-q^n)^5
}
{
(1-q^{5n})^3
}
\left(
1+
\frac{1-\sqrt{5}}{2} 
q^n
+
q^{2n} 
\right)^5   \\
=&
1
-
\frac52(1+\sqrt{5})
\sum_{n=1}^{\infty} (d_{1,5}(n)-d_{4,5}(n)) q^n 
+
\frac52
(7+3 \sqrt{5})
\sum_{n=1}^{\infty} (d_{2,5}(n)-d_{3,5}(n)) q^n,   \\
W(q)=&
\prod_{n=1}^{\infty}
\frac
{
(1-q^n)^5
}
{
(1-q^{5n})^3
}
\left(
1+
\frac{1+\sqrt{5}}{2} 
q^n
+
q^{2n} 
\right)^5  \\
=&
1
+
\frac52(-1+\sqrt{5})
\sum_{n=1}^{\infty} (d_{1,5}(n)-d_{4,5}(n)) q^n 
+
\frac52
(7-3 \sqrt{5})
\sum_{n=1}^{\infty} (d_{2,5}(n)-d_{3,5}(n)) q^n. 
\end{align*}
Then we have 
\begin{align*}
E_4(q)=&
\frac{\left(123-55 \sqrt{5}\right) V^4+12 \left(-11+5 \sqrt{5}\right) V^3 W+28 V^2 W^2-12 \left(11+5 \sqrt{5}\right) V W^3+\left(123+55 \sqrt{5}\right) W^4  }{10},   
 \\
E_6(q)=&
\frac{1}{25} \left(-1525+682 \sqrt{5}\right) V^6-\frac{9}{25} \left(-275+123 \sqrt{5}\right) V^5 W
+\frac{3}{2} \left(-25+11 \sqrt{5}\right) V^4 W^2  \\
&-\frac{3}{2} \left(25+11 \sqrt{5}\right) V^2 W^4
+\frac{9}{25} \left(275+123 \sqrt{5}\right) V W^5-\frac{1}{25}
   \left(1525+682 \sqrt{5}\right) W^6,
\end{align*}
and 
\begin{align*}
q(q;q)_{\infty}^{24}
=&
-\frac{1}{250} V^5 W^5 (V-W) \left(\left(-11+5 \sqrt{5}\right) V+\left(11+5 \sqrt{5}\right) W\right)  \\
=&
\frac{1}{250} \left(11-5 \sqrt{5}\right) V^7 W^5-\frac{11 }{125}V^6 W^6+\frac{1}{250} \left(11+5 \sqrt{5}\right) V^5 W^7. 
\end{align*}
}
\end{corollary}

\begin{theorem}
\label{thm:G4-G6-(1,1/5)-(1,3/5)}
{\it
For every $\tau\in\mathbb{H}^2,$ we have 
\begin{equation*}
G_4(\tau)
=
\frac
{
\theta^{\prime  }
\left[
\begin{array}{c}
1 \\
1
\end{array}
\right]^4
}
{
720\cdot
\theta^{12}
\left[
\begin{array}{c}
 1 \\
\frac15
\end{array}
\right]
\theta^{12}
\left[
\begin{array}{c}
1 \\
\frac35
\end{array}
\right]
}
\left(A^4-12 A^3 B+14 A^2 B^2+12 A B^3+B^4\right), 
\end{equation*}
and
\begin{align*}
G_6(\tau)=&
\frac
{-
\theta^{\prime  }
\left[
\begin{array}{c}
1 \\
1
\end{array}
\right]^6
}
{
30240 \cdot
\theta^{18}
\left[
\begin{array}{c}
 1 \\
\frac15
\end{array}
\right]
\theta^{18}
\left[
\begin{array}{c}
1 \\
\frac35
\end{array}
\right]
}  
\left(A^2+B^2\right) \left(A^4-18 A^3 B+74 A^2 B^2+18 A B^3+B^4\right),
\end{align*}
where 
\begin{equation*}
A=\theta^{5}
\left[
\begin{array}{c}
 1 \\
\frac15
\end{array}
\right], \,\,
B=
\theta^{5}
\left[
\begin{array}{c}
 1 \\
\frac35
\end{array}
\right]. 
\end{equation*}
}
\end{theorem}

\begin{proof}
The theorem is based on equations (\ref{eqn:g2-X-Y(1)}) and (\ref{eqn:g3-X-Y(1)}) and the derivative formulas of Theorem \ref{thm:derivative-level-5-(1,1/5)-(1,3/5)}. 
\end{proof}

\begin{theorem}
\label{thm:D-j-(1,1/5)-(1,3/5)}
{\it
For every $\tau\in\mathbb{H}^2,$ we have 
\begin{equation*}
\Delta=g_2^3-27g_3^2
=
\frac
{
\theta^{\prime  }
\left[
\begin{array}{c}
1 \\
1
\end{array}
\right]^{12}
}
{
\theta^{11}
\left[
\begin{array}{c}
 1 \\
\frac15
\end{array}
\right]
\theta^{11}
\left[
\begin{array}{c}
1 \\
\frac35
\end{array}
\right]
}  
\left( 
\theta^{10}
\left[
\begin{array}{c}
 1 \\
\frac15
\end{array}
\right]
-11
\theta^{5}
\left[
\begin{array}{c}
 1 \\
\frac15
\end{array}
\right]
\theta^{5}
\left[
\begin{array}{c}
1 \\
\frac35
\end{array}
\right]
-
\theta^{10}
\left[
\begin{array}{c}
1 \\
\frac35
\end{array}
\right]
\right),
\end{equation*}
and
\begin{equation*}
j=\frac
{
12^3 g_2^3
}
{
\Delta
}
=
\frac
{
\left(f^4(\tau)-12 f^3(\tau)+14 f^2(\tau) +12f(\tau)+1\right)^3
}
{
f^5(\tau) (f^2(\tau)-11f(\tau)-1)
},
\end{equation*}
where 
\begin{equation*}
f(\tau)=\theta^{5}
\left[
\begin{array}{c}
 1 \\
\frac15
\end{array}
\right]
/
\theta^{5}
\left[
\begin{array}{c}
 1 \\
\frac35
\end{array}
\right]
=\frac{11+5\sqrt{5}}{2}
\prod_{n=1}^{\infty}
\frac{ \left(1+\frac{1+\sqrt{5}}{2}q^n+q^{2n} \right)^5 }{ \left(1+\frac{1-\sqrt{5}}{2}q^n+q^{2n} \right)^5 }, \,\,q=\exp(2\pi i \tau). 
\end{equation*}
}
\end{theorem}

\begin{proof}
By using equations (\ref{eqn:g2-X-Y(1)}) and (\ref{eqn:g3-X-Y(1)}), 
we have 
\begin{align*}
\Delta=&-25 (3 X-Y)^5 (X+3 Y)^5 \left(X^2+4 X Y-Y^2\right)
=
\frac
{
\theta^{\prime  }
\left[
\begin{array}{c}
1 \\
1
\end{array}
\right]^{12}
}
{
\theta^{36}
\left[
\begin{array}{c}
 1 \\
\frac15
\end{array}
\right]
\theta^{36}
\left[
\begin{array}{c}
1 \\
\frac35
\end{array}
\right]
}  
A^5B^5 (A^2-11AB-B^2),  \\
j=&
\frac{320 \left(X^4-3 X^3 Y-31 X^2 Y^2+3 X Y^3+Y^4\right)^3}{27 (3 X-Y)^5 (X+3 Y)^5 \left(X^2+4 X Y-Y^2\right)} 
=
\frac
{
\left(A^4-12 A^3 B+14 A^2 B^2+12 A B^3+B^4\right)^3
}
{
A^5B^5 (A^2-11AB-B^2)
},
\end{align*}
where 
\begin{equation*}
A=\theta^{5}
\left[
\begin{array}{c}
 1 \\
\frac15
\end{array}
\right], \,\,
B=
\theta^{5}
\left[
\begin{array}{c}
 1 \\
\frac35
\end{array}
\right]. 
\end{equation*}
\end{proof}

\begin{corollary}
\label{coro:j-(P,Q)-(1,1/5)-(1,3/5)}
{\it
For $q\in\mathbb{C}$ with $|q|<1,$ 
set 
\begin{align*}
P(q)=&\cot \frac{2 \pi}{5} + 4 \sin \frac{4 \pi}{5} \sum_{n=1}^{\infty} (d_{1,5}(n)-d_{4,5}(n)) q^n
-4\sin \frac{2\pi}{5} \sum_{n=1}^{\infty} (d_{2,5}(n)-d_{3,5}(n)) q^n, \\
Q(q)=&\cot \frac{ \pi}{5} + 4 \sin \frac{2 \pi}{5} \sum_{n=1}^{\infty} (d_{1,5}(n)-d_{4,5}(n)) q^n
+4\sin \frac{4\pi}{5} \sum_{n=1}^{\infty} (d_{2,5}(n)-d_{3,5}(n)) q^n. 
\end{align*}
Then, we have 
\begin{equation*}
j(\tau)=j(q)=
\frac{320 \left(P^4-3 P^3 Q-31 P^2 Q^2+3 P Q^3+Q^4\right)^3}{27 (3 P-Q)^5 (P+3 Q)^5 \left(P^2+4 P Q-Q^2\right)}, 
\,\
\end{equation*}
and
\begin{equation*}
f(\tau)=\theta^{5}
\left[
\begin{array}{c}
 1 \\
\frac15
\end{array}
\right]
/
\theta^{5}
\left[
\begin{array}{c}
 1 \\
\frac35
\end{array}
\right]
=
-
\frac{P+3Q }{3P-Q },  \,\,q=\exp(2\pi i \tau). 
\end{equation*}
}
\end{corollary}

\begin{proof}
The expression of $f(\tau)$ follows from the derivative formulas (\ref{eqn:analogue-Jacobi-(1,1/5)}) and (\ref{eqn:analogue-Jacobi-(1,3/5)}), which implies that 
\begin{equation*}
f(\tau)=-\frac{X+3Y}{3X-Y}. 
\end{equation*}
\end{proof}

\begin{theorem}
\label{thm:another-Jacobi-(1,1/5)-(1,3/5)}
{\it
For every $\tau\in\mathbb{H}^2,$ we have  
\begin{equation}
\label{eqn:Jacobi-(1,1/5)-(1,3/5)}
\theta^{\prime}
\left[
\begin{array}{c}
1 \\
1
\end{array}
\right]^4
=
\frac
{(2 \pi)^4
\theta^{11}
\left[
\begin{array}{c}
1 \\
\frac15
\end{array}
\right]
\theta^{11}
\left[
\begin{array}{c}
1 \\
\frac35
\end{array}
\right]
}
{
\theta^{10}
\left[
\begin{array}{c}
1 \\
\frac15
\end{array}
\right]
-
11
\theta^{5}
\left[
\begin{array}{c}
1 \\
\frac15
\end{array}
\right]
\theta^{5}
\left[
\begin{array}{c}
1 \\
\frac35
\end{array}
\right]
-
\theta^{10}
\left[
\begin{array}{c}
1 \\
\frac35
\end{array}
\right]
}.
\end{equation}
}
\end{theorem}

\begin{proof}
Recall the following formula from Apostol \cite[pp. 51]{Apostol}: 

\begin{equation}
\label{eqn:Delta-eta}
\Delta=(2\pi)^{12} \eta^{24}(\tau)=(2\pi)^4
\theta^{\prime}
\left[
\begin{array}{c}
1 \\
1
\end{array}
\right]^8. 
\end{equation}
The theorem follows from Theorem \ref{thm:D-j-(1,1/5)-(1,3/5)}. 
\end{proof}

\begin{theorem}
\label{thm:G4-G6-modular5-(1)}
{\it
\noindent
\begin{enumerate}\itemsep=0pt
\item[$\mathrm{(1)}$] For each positive integer $k,$ 
the Eisenstein series $G_{4k}(\tau)$ can be expressed using a rational expression as follows: 
\begin{equation*}
\theta
\left[
\begin{array}{c}
1 \\
\frac15
\end{array}
\right](0,\tau), \,\,
\text{and}   \,\,
\theta
\left[
\begin{array}{c}
1 \\
\frac35
\end{array}
\right](0,\tau).
\end{equation*} 
\item[$\mathrm{(2)}$] For each positive integer $k,$ 
the Eisenstein series $G_{4k+2}(\tau)$ can be expressed using a rational expression as follows: 
\begin{equation*}
\theta
\left[
\begin{array}{c}
1 \\
\frac15
\end{array}
\right](0,\tau), \,\,
\theta
\left[
\begin{array}{c}
1 \\
\frac35
\end{array}
\right](0,\tau), \,\,
\text{and} \,\,
\theta^{\prime}
\left[
\begin{array}{c}
1 \\
1
\end{array}
\right](0,\tau).
\end{equation*} 
\end{enumerate}
}
\end{theorem}

\begin{proof}
From the study by Apostol \cite[pp. 118]{Apostol}, 
we recall that for $k=2, 3, 4\ldots,$ 
\begin{equation*}
G_{2k}(\tau)
=
\sum_{a,b} c_{a,b} G_4^4(\tau)G^b_6(\tau), 
\end{equation*}
where $c_{a,b}$ are complex numbers, and the sum is extended over all integers $a\ge 0, b\ge 0$ such that $4a+6b=k.$ 
The theorem follows from Theorems \ref{thm:G4-G6-(1,1/5)-(1,3/5)} and \ref{thm:another-Jacobi-(1,1/5)-(1,3/5)}. 
\end{proof}

\subsection{Proof of Ramanujan's ODEs (\ref{eqn:Ramanujan-ODE})}

\begin{theorem}
{\it
For $q\in\mathbb{C}$ with $|q|<1,$ 
we have 
\begin{equation*}
q\frac{d E_2}{dq}=\frac{(E_2)^2-E_4  }{12  }, \,\,
q\frac{ dE_4}{dq}=\frac{E_2 E_4-E_6  }{3  }, \,\,
q\frac{dE_6}{dq}=\frac{E_2 E_6-(E_4)^2  }{2  }. 
\end{equation*}
}
\end{theorem}

\begin{proof}
According to Theorems \ref{thm-level5-E_2(q)} and \ref{thm-E4-E6-level5-q}, 

\begin{equation*}
q\frac{d}{dq}E_2=q\frac{d}{dq}R=\frac{5P^4-15P^3Q-155P^2Q^2+15PQ^3+5Q^4+R^2}{12}
=\frac{E_2^2-E_4}{12}. 
\end{equation*}
Next, it is observed that 
\begin{align*}
q\frac{d}{dq}E_4=&
q\frac{d}{dq}(-5P^4+15P^3Q+155P^2Q^2-15PQ^3-5Q^4)  \\
=&
\frac{395 P^6}{24}+\frac{615 P^5 Q}{4}+\frac{2675 P^4 Q^2}{8}+\frac{2675 P^2 Q^4}{8}-\frac{615 P Q^5}{4}+\frac{395Q^6}{24}\\
&-\frac{5 P^4 R}{3}+5 P^3 Q R+\frac{155}{3} P^2 Q^2 R
-5 P Q^3 R-\frac{5 Q^4 R}{3}  \\
=&\frac{E_2E_4-E_6}{3}. 
\end{align*}
Finally, we obtain 
\begin{align*}
&q\frac{d}{dq}E_6=
q\frac{d}{dq}
\left\{
\frac
{
-395P^6-3690P^5Q-8025P^4Q^2-8025P^2Q^4+3690PQ^5-395Q^6
}
{
8
}\right\}   \\  
=&
-\frac{25 P^8}{2}+75 P^7 Q+\frac{1325 P^6 Q^2}{2}-2400 P^5 Q^3-\frac{23625 P^4 Q^4}{2}  \\
&+2400 P^3Q^5+\frac{1325 P^2 Q^6}{2}-75 P Q^7-\frac{25 Q^8}{2}\\
&-\frac{395 P^6 R}{16}-\frac{1845}{8} P^5 Q R-\frac{8025}{16} P^4 Q^2 R-\frac{8025}{16} P^2 Q^4 R+\frac{1845}{8} P Q^5 R-\frac{395 Q^6 R}{16}  \\
=&\frac{E_2E_6-(E_4)^2}{2}. 
\end{align*}
\end{proof}

\section{Proof of Theorem \ref{thm-level5-E_2(q^5)}}
\label{sec:proof:level5-E_2(q^5)}

\subsection{Notations}
For every $\tau\in\mathbb{H}^2,$ set $y=\exp(2\pi i \tau/5)$ and 
\begin{equation*}
X=
\frac
{
\theta^{\prime}
\left[
\begin{array}{c}
\frac15 \\
1
\end{array}
\right]
}
{
\theta
\left[
\begin{array}{c}
\frac15 \\
1
\end{array}
\right]
},  \,\,
Y=
\frac
{
\theta^{\prime}
\left[
\begin{array}{c}
\frac35 \\
1
\end{array}
\right]
}
{
\theta
\left[
\begin{array}{c}
\frac35 \\
1
\end{array}
\right]
},  \,\,
Z=
\frac
{
\theta^{\prime  \prime  \prime}
\left[
\begin{array}{c}
1 \\
1
\end{array}
\right]
}
{
\theta^{\prime  }
\left[
\begin{array}{c}
1 \\
1
\end{array}
\right]
}. 
\end{equation*}

\subsection{ ODEs for $P(q)$ and $Q(q)$}

\begin{proof}
According to Proposition \ref{prop:2nd-derivative-(1/5,1)-(3/5,1)}, we have 
\begin{equation}
\frac{
\theta^{\prime \prime}
\left[
\begin{array}{c}
\frac15 \\
1
\end{array}
\right]  
}
{ 
\theta
\left[
\begin{array}{c}
\frac15\\
1
\end{array}
\right]
}
=
\frac13 Z
-4XY-\frac43X^2+\frac23Y^2,   \,\,
\frac{
\theta^{\prime \prime}
\left[
\begin{array}{c}
\frac35 \\
1
\end{array}
\right]  
}
{ 
\theta
\left[
\begin{array}{c}
\frac35 \\
1
\end{array}
\right]
}
=
\frac13 Z
+4XY+\frac23X^2-\frac43Y^2.  
\end{equation}
According to Proposition \ref{prop:3rd-derivative-(1/5,1)-(3/5,1)}, we obtain
\begin{equation}
\frac{
\theta^{\prime \prime \prime}
\left[
\begin{array}{c}
\frac15\\
1
\end{array}
\right]  
}
{ 
\theta
\left[
\begin{array}{c}
\frac15 \\
1
\end{array}
\right]
}
=
3X^3+9X^2Y-15XY^2+3Y^3+XZ,   \,\,
\frac{
\theta^{\prime \prime \prime}
\left[
\begin{array}{c}
\frac35\\
1
\end{array}
\right]  
}
{ 
\theta
\left[
\begin{array}{c}
\frac35 \\
1
\end{array}
\right]
}
=
-3X^3-15X^2Y-9XY^2+3Y^3+YZ.  
\end{equation}
\par
Because 
\begin{equation*}
\frac{
\theta^{\prime \prime \prime}
\left[
\begin{array}{c}
\frac{j}{5} \\
1
\end{array}
\right]  
}
{ 
\theta
\left[
\begin{array}{c}
\frac{j}{5} \\
1
\end{array}
\right]
}
=
4 \pi i
\frac{d}{d\tau}
\left\{
\frac{
\theta^{\prime }
\left[
\begin{array}{c}
\frac{j}{5} \\
1
\end{array}
\right]  
}
{ 
\theta
\left[
\begin{array}{c}
\frac{j}{5} \\
1
\end{array}
\right]
}
\right\}
+
\frac{
\theta^{\prime }
\left[
\begin{array}{c}
\frac{j}{5} \\
1
\end{array}
\right]  
}
{ 
\theta
\left[
\begin{array}{c}
\frac{j}{5} \\
1
\end{array}
\right]
}
\cdot
\frac{
\theta^{\prime \prime }
\left[
\begin{array}{c}
\frac{j}{5}\\
1
\end{array}
\right]  
}
{ 
\theta
\left[
\begin{array}{c}
\frac{j}{5}\\
1
\end{array}
\right]
},  \,\,(j=1,3), 
\end{equation*}
we get 
\begin{align*}
4\pi i X^{\prime}=&
\frac{13}{3}X^3+13X^2Y-\frac{47}{3}XY^2+3Y^3+\frac23XZ,  \\
4\pi i Y^{\prime}=&
-3X^3-\frac{47}{3}X^2Y-13XY^2+\frac{13}{3}Y^3+\frac23YZ, 
\end{align*}
where ${ }^{\prime}=d/d\tau. $ 
Thus, by changing $\tau\rightarrow 5\tau,$ we obtain the ODEs for $P(q)$ and $Q(q)$. 
\end{proof}

\subsection{ODE for $R(q)$}

\begin{proof}
First, note that
\begin{align*}
\frac{
\theta^{(4)}
\left[
\begin{array}{c}
\frac15\\
1
\end{array}
\right]  
}
{ 
\theta
\left[
\begin{array}{c}
\frac15 \\
1
\end{array}
\right]
}
=&
4 \pi i
\frac{d}{d\tau}
\left\{
\frac{
\theta^{\prime \prime }
\left[
\begin{array}{c}
\frac15 \\
1
\end{array}
\right]  
}
{ 
\theta
\left[
\begin{array}{c}
\frac15 \\
1
\end{array}
\right]
}
\right\}
+
\left\{
\frac{
\theta^{\prime \prime}
\left[
\begin{array}{c}
\frac15 \\
1
\end{array}
\right]  
}
{ 
\theta
\left[
\begin{array}{c}
\frac15 \\
1
\end{array}
\right]
}
\right\}^2,   \\
=&
\frac{4\pi i}{3}\frac{dZ}{d\tau}
+\frac{20}{9}X^4+\frac{52}{3}X^3Y+\frac{316}{9}X^2Y^2+\frac{44}{3}XY^3-\frac{52}{9}Y^4  \\
&\quad -8XYZ-\frac83X^2Z+\frac43Y^2Z+\frac19Z^2,  
\end{align*}
and 
\begin{align*}
\frac{
\theta^{(5)}
\left[
\begin{array}{c}
1 \\
1
\end{array}
\right]  
}
{ 
\theta^{(1)}
\left[
\begin{array}{c}
1 \\
1
\end{array}
\right]
}
=&
4 \pi i
\frac{d}{d\tau}
\left\{
\frac{
\theta^{(3) }
\left[
\begin{array}{c}
1 \\
1
\end{array}
\right]  
}
{ 
\theta^{(1)}
\left[
\begin{array}{c}
1 \\
1
\end{array}
\right]
}
\right\}
+
\left\{
\frac{
\theta^{(3)}
\left[
\begin{array}{c}
1 \\
1
\end{array}
\right]  
}
{ 
\theta^{(1)}
\left[
\begin{array}{c}
1 \\
1
\end{array}
\right]
}
\right\}^2   
=
4\pi i \frac{dZ}{d\tau}+Z^2.
\end{align*}
Thus, according to Proposition \ref{prop:4th-derivative-(1/5,1)-(3/5,1)}, 
we have 
\begin{equation*}
4\pi i Z^{\prime}=
\frac{10}{3}X^4-10X^3Y-\frac{310}{3}X^2Y^2+10XY^3+\frac{10}{3}Y^4+\frac23 Z^2,
\end{equation*}
and
\begin{equation*}
\frac{
\theta^{(4)}
\left[
\begin{array}{c}
\frac15 \\
1
\end{array}
\right]  
}
{ 
\theta
\left[
\begin{array}{c}
\frac15\\
1
\end{array}
\right]
}
=
\frac{10}{3}X^4+14X^3Y+\frac23 X^2Y^2+18XY^3-\frac{14}{3}XY^3+\frac13Z^2-8XYZ-\frac83 X^2Z+\frac43 Y^2Z, 
\end{equation*}
where ${ }^{\prime}=d/d\tau.$ 
Thus, by changing $\tau\rightarrow 5\tau,$ we obtain the ODE for $R(q)$.
\end{proof}

\begin{theorem}
\label{thm-E4-E6-level5-q^5}
{\it
For $q\in\mathbb{C}$ with $|q|<1,$ 
set 
\begin{equation*}
P(q)=1+ 10 \sum_{n=1}^{\infty} (d_{2,5}(n)-d_{3,5}(n)) q^n, \,\,
Q(q)=3 + 10\sum_{n=1}^{\infty} (d_{1,5}(n)-d_{4,5}(n)) q^n. 
\end{equation*}
Then, we have 
\begin{align*}
E_4(q^5)=&
\frac{
-P^4+3P^3Q+31P^2Q^2-3PQ^3-Q^4
}
{125}
, \\
E_6(q^5)=&
\frac
{
79P^6+738P^5Q+1605P^4Q^2+1605P^2Q^4-738PQ^5+79Q^6
}
{
25000
}. 
\end{align*}
}
\end{theorem}

\begin{proof}
For the proof, we use 
\begin{equation*}
\wp=\wp(z;1,\tau)
=
\frac13
\frac
{
\theta^{\prime  \prime  \prime}
\left[
\begin{array}{c}
1 \\
1
\end{array}
\right]
}
{
\theta^{\prime  }
\left[
\begin{array}{c}
1 \\
1
\end{array}
\right]
}
-
\frac{d^2}{dz^2}
\log
\theta
\left[
\begin{array}{c}
1 \\
1
\end{array}
\right](z). 
\end{equation*} 
\par
First, note that 
\begin{equation*}
\wp^{\prime \prime}=6\wp^2-g_2, \,\,
g_2(1,\tau)=\frac{4\pi^4}{3}E_4(q).
\end{equation*}
Then, by substituting $z=-2\tau/5,$ we have 
\begin{equation}
\label{eqn:g2-X-Y(2)}
g_2=-\frac{20}{3}X^4+20X^3Y+\frac{620}{3}X^2Y^2-20XY^3-\frac{20}{3}Y^4, 
\end{equation}
which implies 
\begin{equation*}
E_4(q^5)=\frac{
-P^4+3P^3Q+31P^2Q^2-3PQ^3-Q^4
}
{125}. 
\end{equation*}
\par
Next, it is observed that 
\begin{equation*}
(\wp^{\prime})^2=4\wp^3-g_2\wp-g_3, \,\,
g_3(1,\tau)=\frac{8\pi^6}{27}E_6(q).
\end{equation*}
Thus, by substituting $z=-2\tau/5,$ we obtain 
\begin{equation}
\label{eqn:g3-X-Y(2)}
g_3=-\frac{395}{27}X^6-\frac{410}{3}X^5Y-\frac{2675}{9}X^4Y^2
-\frac{2675}{9}X^2Y^4+\frac{410}{3}XY^5-\frac{395}{27}Y^6, 
\end{equation}
which implies 
\begin{equation*}
E_6(q^5)=\frac
{
79P^6+738P^5Q+1605P^4Q^2+1605P^2Q^4-738PQ^5+79Q^6
}
{
25000
}. 
\end{equation*}
\end{proof}

\begin{corollary}
\label{coro-E4-E6-level5-q^5-(S,T)}
{\it
For $q\in\mathbb{C}$ with $|q|<1,$ 
set 
\begin{align*}
S(q)=&
q
\prod_{n=1}^{\infty} 
\frac
{
(1-q^n)^2
}
{
(1-q^{5n-2})^5(1-q^{5n-3})^5    
}
=
\sum_{n=1}^{\infty} (d_{1,5}(n)-d_{4,5}(n)) q^n
-3
\sum_{n=1}^{\infty} (d_{2,5}(n)-d_{3,5}(n)) q^n,  
\end{align*}
and 
\begin{align*}
T(q)=&
\prod_{n=1}^{\infty} 
\frac
{
(1-q^n)^2
}
{
(1-q^{5n-1})^5(1-q^{5n-4})^5    
}
=
1+
3
\sum_{n=1}^{\infty} (d_{1,5}(n)-d_{4,5}(n)) q^n
+
\sum_{n=1}^{\infty} (d_{2,5}(n)-d_{3,5}(n)) q^n.
\end{align*}
Then, we have 
\begin{align*}
E_4(q^5)=&
S^4+12 S^3 T+14 S^2 T^2-12 S T^3+T^4,  \\
E_6(q^5)=&
\left(S^2+T^2\right) \left(S^4+18 S^3 T+74 S^2 T^2-18 S T^3+T^4\right),
\end{align*}
and
\begin{align*}
q^5(q^5;q^5)_{\infty}=&-S^5 T^5 \left(S^2+11 S T-T^2\right)  \\
j(q^5)=&
-\frac{\left(S^4+12 S^3 T+14 S^2 T^2-12 S T^3+T^4\right)^3}{S^5 T^5 \left(S^2+11 S T-T^2\right)} 
=
-\frac{\left(g^4+12 g^3+14 g^2-12 g+1\right)^3}{g^5 \left(g^2+11 g-1\right)},  \,\,
g=S/T. 
\end{align*}
}
\end{corollary}

\begin{proof}
The corollary can be obtained by considering 
\begin{equation*}
S=-\frac{3}{10}P+\frac{1}{10} Q, \,\,
T=\frac{1}{10} P+\frac{3}{10} Q. 
\end{equation*}
\end{proof}

\begin{theorem}
\label{thm:G4-G6-(1/5,1)-(3/5,1)}
{\it
For every $\tau\in\mathbb{H}^2,$ we have 
\begin{equation*}
G_4(\tau)
=
\frac
{
\zeta_5^2
\theta^{\prime  }
\left[
\begin{array}{c}
1 \\
1
\end{array}
\right]^4
}
{
720\cdot
\theta^{12}
\left[
\begin{array}{c}
\frac15 \\
1
\end{array}
\right]
\theta^{12}
\left[
\begin{array}{c}
\frac35 \\
1
\end{array}
\right]
}
 \left(A^4+12 A^3 B+14 A^2 B^2-12 A B^3+B^4\right), 
\end{equation*}
and
\begin{align*}
G_6(\tau)=&
\frac
{-\zeta_5^3
\theta^{\prime  }
\left[
\begin{array}{c}
1 \\
1
\end{array}
\right]^6
}
{
30240 \cdot
\theta^{18}
\left[
\begin{array}{c}
 \frac15 \\
1
\end{array}
\right]
\theta^{18}
\left[
\begin{array}{c}
\frac35 \\
1
\end{array}
\right]
}  
 \left(A^2+B^2\right) \left(A^4+18 A^3 B+74 A^2 B^2-18 A B^3+B^4\right),
\end{align*}
where $\zeta_5=\exp(2\pi i/5),$ and 
\begin{equation*}
A=\theta^{5}
\left[
\begin{array}{c}
\frac15\\
1
\end{array}
\right], \,\,
B=
\theta^{5}
\left[
\begin{array}{c}
\frac35 \\
1
\end{array}
\right]. 
\end{equation*}
}
\end{theorem}

\begin{proof}
The theorem follows from equations (\ref{eqn:g2-X-Y(2)}) and (\ref{eqn:g3-X-Y(2)}) and the derivative formulas of Theorem \ref{thm:derivative-level-5-(1/5,1)-(3/5,1)}. 
\end{proof}

\begin{theorem}
\label{thm:D-j-(1/5,1)-(3/5,1)}
{\it
For every $\tau\in\mathbb{H}^2,$ we have 
\begin{equation*}
\Delta(\tau)=g_2^3-27g_3^2
=
\frac
{\zeta_5
\theta^{\prime  }
\left[
\begin{array}{c}
1 \\
1
\end{array}
\right]^{12}
}
{
\theta^{11}
\left[
\begin{array}{c}
\frac15 \\
1
\end{array}
\right]
\theta^{11}
\left[
\begin{array}{c}
\frac35 \\
1
\end{array}
\right]
}  
\left(
\theta^{10}
\left[
\begin{array}{c}
\frac35 \\
1
\end{array}
\right]
-
11
\theta^{5}
\left[
\begin{array}{c}
\frac35 \\
1
\end{array}
\right]
\theta^{5}
\left[
\begin{array}{c}
\frac15 \\
1
\end{array}
\right]
-
\theta^{10}
\left[
\begin{array}{c}
\frac15 \\
1
\end{array}
\right]
\right),
\end{equation*}
and
\begin{equation*}
j(5\tau)=\frac
{
12^3 g_2(5\tau)^3
}
{
\Delta(5\tau)
}
=
\frac
{
\left(g^4(\tau)-12 g^3(\tau)+14 g^2(\tau) +12 g(\tau)+1\right)^3
}
{
g^5(\tau) (g^2(\tau)-11g(\tau)-1)
},
\end{equation*}
where $\zeta_5=\exp(2\pi i/5),$ 
\begin{equation*}
g(\tau)
=
\theta^{5}
\left[
\begin{array}{c}
\frac35 \\
1
\end{array}
\right](0,5\tau)
/
\theta^{5}
\left[
\begin{array}{c}
\frac15 \\
1
\end{array}
\right](0,5\tau)
=
-
q
\prod_{n=1}^{\infty}\frac{(1-q^{5n-1})^5(1-q^{5n-4})^5  }{(1-q^{5n-2})^5(1-q^{5n-3})^5  }, \,\,
q=\exp(2 \pi i \tau). 
\end{equation*}
}
\end{theorem}

\begin{proof}
Based on equations (\ref{eqn:g2-X-Y(2)}) and (\ref{eqn:g3-X-Y(2)}), 
we have 
\begin{align*}
\Delta=&-25 (3 X-Y)^5 (X+3 Y)^5 \left(X^2+4 X Y-Y^2\right)
=
\frac
{-\zeta_5
\theta^{\prime  }
\left[
\begin{array}{c}
1 \\
1
\end{array}
\right]^{12}
}
{
\theta^{36}
\left[
\begin{array}{c}
\frac15 \\
1
\end{array}
\right]
\theta^{36}
\left[
\begin{array}{c}
\frac35 \\
1
\end{array}
\right]
}  
A^5B^5 (A^2+11AB-B^2),  \\
j=&
\frac{320 \left(X^4-3 X^3 Y-31 X^2 Y^2+3 X Y^3+Y^4\right)^3}{27 (3 X-Y)^5 (X+3 Y)^5 \left(X^2+4 X Y-Y^2\right)}
=
-
\frac
{
\left(A^4+12 A^3 B+14 A^2 B^2-12 A B^3+B^4\right)^3
}
{
A^5B^5 (A^2+11AB-B^2)
},
\end{align*}
where 
\begin{equation*}
A=\theta^{5}
\left[
\begin{array}{c}
\frac15 \\
1
\end{array}
\right], \,\,
B=
\theta^{5}
\left[
\begin{array}{c}
\frac35 \\
1
\end{array}
\right]. 
\end{equation*}
\end{proof}

\begin{corollary}
\label{coro:j-(P,Q)-(1/5,1)-(3/5,1)}
{\it
For $q\in\mathbb{C}$ with $|q|<1,$ 
set 
\begin{equation*}
P(q)=1+ 10 \sum_{n=1}^{\infty} (d_{2,5}(n)-d_{3,5}(n)) q^n, \,\,
Q(q)=3 + 10\sum_{n=1}^{\infty} (d_{1,5}(n)-d_{4,5}(n)) q^n. 
\end{equation*}
Then, we have 
\begin{equation*}
j(5\tau)=j(q^5)=
\frac{320 \left(P^4-3 P^3 Q-31 P^2 Q^2+3 P Q^3+Q^4\right)^3}{27 (3 P-Q)^5 (P+3 Q)^5 \left(P^2+4 P Q-Q^2\right)}, 
\end{equation*}
and
\begin{equation*}
g(\tau)
=
\theta^{5}
\left[
\begin{array}{c}
\frac35 \\
1
\end{array}
\right](0,5\tau)
/
\theta^{5}
\left[
\begin{array}{c}
\frac15 \\
1
\end{array}
\right](0,5\tau)
=
\frac{3P-Q}{P+3Q}, \,\,  q=\exp(2\pi i \tau). 
\end{equation*}
}
\end{corollary}

\begin{proof}
The expression of $g(\tau)$ follows from the derivative formulas (\ref{eqn:analogue-Jacobi-(1/5,1)}) and (\ref{eqn:analogue-Jacobi-(3/5,1)}), 
which imply that 
\begin{equation*}
g(\tau/5)=\frac{3X-Y}{X+3Y}.
\end{equation*}
\end{proof}

\begin{theorem}
\label{thm:f-g-relation}
{\it
For every $\tau\in\mathbb{H}^2,$ 
set 
\begin{equation*}
f(\tau)=
\frac{
\theta^{5}
\left[
\begin{array}{c}
 1 \\
\frac15
\end{array}
\right](0,\tau)
}
{
\theta^{5}
\left[
\begin{array}{c}
 1 \\
\frac35
\end{array}
\right](0,\tau)
}
=\frac{11+5\sqrt{5}}{2}
\prod_{n=1}^{\infty}
\frac{ \left(1+\frac{1+\sqrt{5}}{2}q^n+q^{2n} \right)^5 }{ \left(1+\frac{1-\sqrt{5}}{2}q^n+q^{2n} \right)^5 },
\end{equation*}
and
\begin{equation*}
g(\tau)
=
\frac
{
\theta^{5}
\left[
\begin{array}{c}
\frac35 \\
1
\end{array}
\right](0,5\tau)
}
{
\theta^{5}
\left[
\begin{array}{c}
\frac15 \\
1
\end{array}
\right](0,5\tau)
}
=
-
q
\prod_{n=1}^{\infty}\frac{(1-q^{5n-1})^5(1-q^{5n-4})^5  }{(1-q^{5n-2})^5(1-q^{5n-3})^5  }, \,\,  q=\exp(2\pi i \tau). 
\end{equation*}
Then, we have 
\begin{equation}
\label{eqn:f-g-relation}
g(\tau)=
\frac
{
\frac12 (-11+5\sqrt{5})f(\tau)-1
}
{
-f(\tau)-\frac12(-11+5\sqrt{5})
}. 
\end{equation}
}
\end{theorem}

\begin{proof}
We first set 
\begin{equation*}
P_2(q):=1+ 10 \sum_{n=1}^{\infty} (d_{2,5}(n)-d_{3,5}(n)) q^n, \,\,
Q_2(q):=3 + 10\sum_{n=1}^{\infty} (d_{1,5}(n)-d_{4,5}(n)) q^n,
\end{equation*}
which implies that 
\begin{align*}
P_1(q):=&\cot \frac{2 \pi}{5} + 4 \sin \frac{4 \pi}{5} \sum_{n=1}^{\infty} (d_{1,5}(n)-d_{4,5}(n)) q^n
-4\sin \frac{2\pi}{5} \sum_{n=1}^{\infty} (d_{2,5}(n)-d_{3,5}(n)) q^n, \\
=&
-\frac25
\sin \frac{2\pi}{5} P_2(q)
+
\frac25  \sin \frac{4\pi}{5} Q_2(q),  \\
Q_1(q):=&\cot \frac{ \pi}{5} + 4 \sin \frac{2 \pi}{5} \sum_{n=1}^{\infty} (d_{1,5}(n)-d_{4,5}(n)) q^n
+4\sin \frac{4\pi}{5} \sum_{n=1}^{\infty} (d_{2,5}(n)-d_{3,5}(n)) q^n, \\ 
=&
\frac25
\sin \frac{4\pi}{5} P_2(q)
+
\frac25  \sin \frac{2\pi}{5} Q_2(q).
\end{align*}
The theorem is derived from Corollaries \ref{coro:j-(P,Q)-(1,1/5)-(1,3/5)} and \ref{coro:j-(P,Q)-(1/5,1)-(3/5,1)}. 
\end{proof}

\begin{theorem}
\label{eqn:another-Jacobi-(1/5,1)-(3/5,1)}
{\it 
For every $\tau\in\mathbb{H}^2,$ we have 
\begin{equation}
\label{eqn:Jacobi-(1/5,1)-(3/5,1)}
\theta^{\prime}
\left[
\begin{array}{c}
1 \\
1
\end{array}
\right]^4
=
\frac
{-(2\pi)^4
\zeta_5^4
\theta^{11}
\left[
\begin{array}{c}
\frac15 \\
1
\end{array}
\right]
\theta^{11}
\left[
\begin{array}{c}
\frac35 \\
1
\end{array}
\right]
}
{
\theta^{10}
\left[
\begin{array}{c}
\frac15 \\
1
\end{array}
\right]
+
11
\theta^{5}
\left[
\begin{array}{c}
\frac15 \\
1
\end{array}
\right]
\theta^{5}
\left[
\begin{array}{c}
\frac35 \\
1
\end{array}
\right]
-
\theta^{10}
\left[
\begin{array}{c}
\frac35 \\
1
\end{array}
\right]
}, 
\end{equation}
where $\zeta_5=\exp(2 \pi i /5).$ 
}
\end{theorem}

\begin{proof}
The theorem is derived from formula (\ref{eqn:Delta-eta}) and Theorem \ref{thm:D-j-(1/5,1)-(3/5,1)}. 
\end{proof}

\begin{theorem}
\label{thm:G4-G6-modular5-(2)}
{\it
\noindent
\begin{enumerate}\itemsep=0pt
\item[$\mathrm{(1)}$] For each positive integer $k,$ 
the Eisenstein series $G_{4k}(\tau)$ can be expressed using a rational expression 
\begin{equation*}
\theta
\left[
\begin{array}{c}
\frac15 \\
1
\end{array}
\right](0,\tau), \,\,
\text{and}   \,\,
\theta
\left[
\begin{array}{c}
\frac35 \\
1
\end{array}
\right](0,\tau).
\end{equation*} 
\item[$\mathrm{(2)}$] For each positive integer $k,$ 
the Eisenstein series $G_{4k+2}(\tau)$ can be expressed using a rational expression 
\begin{equation*}
\theta
\left[
\begin{array}{c}
\frac15 \\
1
\end{array}
\right](0,\tau), \,\,
\theta
\left[
\begin{array}{c}
\frac35 \\
1
\end{array}
\right](0,\tau), \,\,
\text{and} \,\,
\theta^{\prime}
\left[
\begin{array}{c}
1 \\
1
\end{array}
\right](0,\tau).
\end{equation*} 
\end{enumerate}
}
\end{theorem}

\begin{proof}
The theorem can be proved in the same way as Theorem \ref{thm:G4-G6-modular5-(1)}
\end{proof}

\section{Some product-series identities}
\label{sec:product-series-level5}

\subsection{Some theta functional formulas}

\begin{proposition}
\label{prop:theta-function-identities-level5}
{\it
For every $(z,\tau)\in\mathbb{C}\times\mathbb{H}^2,$ we have 
\begin{align}
&
\theta^2
\left[
\begin{array}{c}
1 \\
\frac35
\end{array}
\right]
\theta
\left[
\begin{array}{c}
1 \\
\frac15
\end{array}
\right](z)
\theta
\left[
\begin{array}{c}
1 \\
\frac95
\end{array}
\right](z)
-
\theta^2
\left[
\begin{array}{c}
1 \\
\frac15
\end{array}
\right]
\theta
\left[
\begin{array}{c}
1 \\
\frac35
\end{array}
\right](z)
\theta
\left[
\begin{array}{c}
1\\
\frac75
\end{array}
\right](z) \notag \\
&\hspace{20mm}
+
\theta
\left[
\begin{array}{c}
1 \\
\frac15
\end{array}
\right]
\theta
\left[
\begin{array}{c}
1 \\
\frac35
\end{array}
\right]
\theta^2 
\left[
\begin{array}{c}
1 \\
1
\end{array}
\right](z)=0,   \label{eqn:level5-(1,1/5),(1,3/5)}
\end{align}
\begin{align}
&
-
\zeta_5^2
\theta^2
\left[
\begin{array}{c}
\frac35 \\
1
\end{array}
\right]
\theta
\left[
\begin{array}{c}
\frac15 \\
1
\end{array}
\right](z)
\theta
\left[
\begin{array}{c}
\frac95\\
1
\end{array}
\right](z)
+
\zeta_5^3
\theta^2
\left[
\begin{array}{c}
\frac15 \\
1
\end{array}
\right]
\theta
\left[
\begin{array}{c}
\frac35 \\
1
\end{array}
\right](z)
\theta
\left[
\begin{array}{c}
\frac75 \\
1
\end{array}
\right](z)  \notag \\
&\hspace{20mm}
+
\theta
\left[
\begin{array}{c}
\frac15 \\
1
\end{array}
\right]
\theta
\left[
\begin{array}{c}
\frac35 \\
1
\end{array}
\right]
\theta^2 
\left[
\begin{array}{c}
1 \\
1
\end{array}
\right](z)=0,   \label{eqn:level5-(1/5,1),(3/5,1)}
\end{align}
where $\zeta_5=\exp(2\pi i/5).$ 
}
\end{proposition}

\begin{proof}
Here, we provide the proof of equation (\ref{eqn:level5-(1,1/5),(1,3/5)}). 
Equation (\ref{eqn:level5-(1/5,1),(3/5,1)}) can be proved in the same manner. 
First, note that 
$
\dim \mathcal{F}_{2}\left[
\begin{array}{c}
0 \\
0
\end{array}
\right]
=
2,
$ 
and 
\begin{equation*}
\theta
\left[
\begin{array}{c}
1 \\
\frac15
\end{array}
\right](z,\tau)
\theta
\left[
\begin{array}{c}
1 \\
\frac95
\end{array}
\right](z,\tau)
, \,
\theta 
\left[
\begin{array}{c}
1 \\
\frac35
\end{array}
\right](z,\tau)
\theta
\left[
\begin{array}{c}
1 \\
\frac75
\end{array}
\right](z,\tau), \,
\theta^2
\left[
\begin{array}{c}
1 \\
1
\end{array}
\right](z,\tau) \in
\mathcal{F}_{2}\left[
\begin{array}{c}
0 \\
0
\end{array}
\right].  
\end{equation*}
Therefore, there exist some complex numbers, $x_1, x_2$, and $ x_3$, which are not all zeros, such that 
 \begin{align*}
 &
 x_1
\theta
\left[
\begin{array}{c}
1 \\
\frac15
\end{array}
\right](z,\tau)
\theta
\left[
\begin{array}{c}
1 \\
\frac95
\end{array}
\right](z,\tau)
+x_2
 \theta
\left[
\begin{array}{c}
1\\
\frac35
\end{array}
\right](z,\tau)
\theta
\left[
\begin{array}{c}
1 \\
\frac75
\end{array}
\right](z,\tau)
+
x_3
\theta^2
\left[
\begin{array}{c}
1 \\
1
\end{array}
\right](z,\tau)=0. 
\end{align*}
Note that in the fundamental parallelogram, 
the zero of 
$
\theta
\left[
\begin{array}{c}
1 \\
\frac15
\end{array}
\right](z),
$ 
$
\theta
\left[
\begin{array}{c}
1 \\
\frac35
\end{array}
\right](z),
$ 
or 
$
\theta
\left[
\begin{array}{c}
1 \\
1
\end{array}
\right](z)
$ 
is $z=2/5,$ $1/5$ or $0.$ 
By substituting $z=2/5, 1/5,$ and $0,$ we have 
\begin{alignat*}{4}
&
&    
&
x_2
\theta
\left[
\begin{array}{c}
1\\
\frac35
\end{array}
\right]  
&    
&+x_3
\theta
\left[
\begin{array}{c}
1 \\
\frac15
\end{array}
\right]  
&  
&=0,  \\
&-x_1
\theta
\left[
\begin{array}{c}
1\\
\frac15
\end{array}
\right]  
&  
&
&
&+x_3
\theta
\left[
\begin{array}{c}
1 \\
\frac35
\end{array}
\right]  
&  
&=0,  \\
&
-
x_1
\theta^2
\left[
\begin{array}{c}
1 \\
\frac15
\end{array}
\right]  
&
-&x_2
\theta^2\left[
\begin{array}{c}
1 \\
\frac35
\end{array}
\right] 
&
&
&
&=0. 
\end{alignat*}
By solving this system of equations, 
we get 
\begin{equation*}
(x_1,x_2,x_3)=\alpha
\left(
\theta^2\left[
\begin{array}{c}
1 \\
\frac35
\end{array}
\right], 
-
\theta^2\left[
\begin{array}{c}
1 \\
\frac15
\end{array}
\right], 
\theta\left[
\begin{array}{c}
1 \\
\frac15
\end{array}
\right] 
\theta\left[
\begin{array}{c}
1\\
\frac53
\end{array}
\right] 
\right) \,\,\,   \text{for some} \,\,\alpha\in\mathbb{C}\setminus\{0\},
\end{equation*}
which proves the proposition. 
\end{proof}

\begin{lemma}
\label{lem:first-order-deri-square}
{\it
For every $(z,\tau)\in\mathbb{C}\times\mathbb{H}^2,$ 
we have 
\begin{equation*}
\left\{
\frac{
\theta^{\prime}
\left[
\begin{array}{c}
\epsilon \\
\epsilon^{\prime}
\end{array}
\right](z)
}
{
\theta
\left[
\begin{array}{c}
\epsilon \\
\epsilon^{\prime}
\end{array}
\right](z)
}
\right\}^2
=
\frac{
\theta^{\prime \prime}
\left[
\begin{array}{c}
\epsilon \\
\epsilon^{\prime}
\end{array}
\right](z)
}
{
\theta
\left[
\begin{array}{c}
\epsilon \\
\epsilon^{\prime}
\end{array}
\right](z)
}
-
\frac
{
d^2
}
{
dz^2
}
\log
\theta
\left[
\begin{array}{c}
\epsilon \\
\epsilon^{\prime}
\end{array}
\right](z). 
\end{equation*}
}
\end{lemma}

\begin{proof}
The lemma can be proved by direct calculation. 
\end{proof}

\begin{theorem}
\label{thm:pro-series-level5}
{\it 
For every $\tau\in\mathbb{H}^2,$ we have 
\begin{align}
\frac{\eta^5(\tau)}{\eta(5\tau)}
=
\frac
{
(q;q)_{\infty}^5
}
{
(q^5; q^5)_{\infty}
}
=&
1-5
\sum_{n=1}^{\infty}
\left(  \sum_{d|n} d \left( \frac{d}{5} \right)         \right) q^n, \label{eqn:pro-series-deg1-(1,1/5),(1,3/5)}  \\
\frac{\eta^5(5\tau)}{\eta(\tau)}
=
q\frac{(q^5;q^5)_{\infty}^5}{(q;q)_{\infty}}
=&
\sum_{n=1}^{\infty}
\left(  \sum_{d|n} \frac{n}{d} \left(\frac{d}{5}\right)    \right) q^n,   \quad q=\exp(2\pi i \tau),\label{eqn:pro-series-deg1-(1/5,1),(3/5,1)}
\end{align}
where for each $m\in\mathbb{N},$ 
\begin{equation*}
\left(
\frac{m}{5}
\right)
=
\begin{cases}
1  &\text{if $m\equiv \pm1 \mod 5,$}\\
-1  &\text{if $m\equiv \pm2 \mod 5,$}\\
0  &\text{if $m\equiv 0 \hspace{3mm}\mod 5. $}  \\
\end{cases}
\end{equation*}
}
\end{theorem}

\begin{proof}
By comparing the coefficients of term $z^2$ in equations (\ref{eqn:level5-(1,1/5),(1,3/5)}) and (\ref{eqn:level5-(1/5,1),(3/5,1)}),  
we have 
\begin{equation}
\label{eqn:coefficients-level5-(1)}
\frac
{
\left\{
\theta^{\prime}
\left[
\begin{array}{c}
1 \\
1
\end{array}
\right]
\right\}^2
}
{
\theta
\left[
\begin{array}{c}
1 \\
\frac15
\end{array}
\right]
\theta
\left[
\begin{array}{c}
1 \\
\frac35
\end{array}
\right]
}
=
\left.
\frac{d^2}{dz^2}
\log
\theta
\left[
\begin{array}{c}
1 \\
\frac15
\end{array}
\right](z)
\right|_{z=0}
-
\left.
\frac{d^2}{dz^2}
\log
\theta
\left[
\begin{array}{c}
1 \\
\frac35
\end{array}
\right](z)
\right|_{z=0}, 
\end{equation}
and
\begin{equation}
\label{eqn:coefficients-level5-(2)}
\zeta_5
\frac
{
\left\{
\theta^{\prime}
\left[
\begin{array}{c}
1 \\
1
\end{array}
\right]
\right\}^2
}
{
\theta
\left[
\begin{array}{c}
\frac15 \\
1
\end{array}
\right]
\theta
\left[
\begin{array}{c}
\frac35 \\
1
\end{array}
\right]
}
=
\left.
\frac{d^2}{dz^2}
\log
\theta
\left[
\begin{array}{c}
\frac35 \\
1
\end{array}
\right](z)
\right|_{z=0}
-
\left.
\frac{d^2}{dz^2}
\log
\theta
\left[
\begin{array}{c}
\frac15 \\
1
\end{array}
\right](z)
\right|_{z=0}. 
\end{equation}
Therefore, 
the theorem can be obtained 
using Jacobi's triple product identity (\ref{eqn:Jacobi-triple}).
\end{proof}

\begin{theorem}
\label{thm:expression-level5(1)}
{\it For $q\in\mathbb{C}$ with $|q|<1,$ set 
\begin{align*}
P(q)=&\cot \frac{2 \pi}{5} + 4 \sin \frac{4 \pi}{5} \sum_{n=1}^{\infty} (d_{1,5}(n)-d_{4,5}(n)) q^n
-4\sin \frac{2\pi}{5} \sum_{n=1}^{\infty} (d_{2,5}(n)-d_{3,5}(n)) q^n, \\
Q(q)=&\cot \frac{ \pi}{5} + 4 \sin \frac{2 \pi}{5} \sum_{n=1}^{\infty} (d_{1,5}(n)-d_{4,5}(n)) q^n
+4\sin \frac{4\pi}{5} \sum_{n=1}^{\infty} (d_{2,5}(n)-d_{3,5}(n)) q^n. 
\end{align*}
Then, we have 
\begin{equation}
\frac
{
(q;q)_{\infty}^5
}
{
(q^5;q^5)_{\infty}
}
=
-
\frac
{
\sqrt{5}
}
{
4
}
(3P-Q)(P+3Q).
\end{equation}
}
\end{theorem}

\begin{proof}
The theorem follows from equations (\ref{eqn:2nd-deri-X-Y-(1,1/5)-(1,3/5)}) and (\ref{eqn:coefficients-level5-(1)}). 
\end{proof}

\begin{theorem}
\label{thm:expression-level5(2)}
{\it For $q\in\mathbb{C}$ with $|q|<1,$ set 
\begin{align*}
P(q)=&1+ 10 \sum_{n=1}^{\infty} (d_{2,5}(n)-d_{3,5}(n)) q^n, \,\,
Q(q)=3 + 10\sum_{n=1}^{\infty} (d_{1,5}(n)-d_{4,5}(n)) q^n. 
\end{align*}
Then, we have 
\begin{equation}
q
\frac
{
(q^5;q^5)_{\infty}^5
}
{
(q;q)_{\infty}
}
=
-
\frac{1}{100}
(3P-Q)(P+3Q),  
\end{equation}
and
\begin{equation}
\frac
{
(q;q)_{\infty}^5
}
{
(q^5;q^5)_{\infty}
}
=
\frac{1}{4} \left(P^2+4 P Q-Q^2\right). 
\end{equation}
}
\end{theorem}

\begin{proof}
The theorem can be proved in the same way as Theorem \ref{thm:expression-level5(1)}. 
\end{proof}

\section{Some theorem of Farkas and Kra}
\label{sec:theorem-Farkas-Kra}

\begin{theorem}
\label{thm:Farkas-Kra-5-(1,1/5)-(1,3/5)}
(Farkas and Kra \cite[pp. 318]{Farkas-Kra})
{\it
For every $\tau\in\mathbb{H}^2,$ we have 
\begin{equation*}
\frac{d}{d \tau} 
\log 
\left(
\frac{\eta(5\tau)}{\eta(\tau)}
\right)
+
\frac{1}{2\pi i \cdot 3}
\left[
\left\{
\frac
{
\theta^{\prime} 
\left[
\begin{array}{c}
1 \\
\frac{1}{5}
\end{array}
\right](0,\tau) 
}
{
\theta
\left[
\begin{array}{c}
1 \\
\frac{1}{5}
\end{array}
\right](0,\tau)
}
\right\}^2
+
\left\{
\frac
{
\theta^{\prime} 
\left[
\begin{array}{c}
1 \\
\frac{3}{5}
\end{array}
\right](0,\tau) 
}
{
\theta
\left[
\begin{array}{c}
1 \\
\frac{3}{5}
\end{array}
\right](0,\tau)
}
\right\}^2
\right]=0. 
\end{equation*}
}
\end{theorem}

\begin{proof}
Summing both sides of equations (\ref{eqn:2nd-deri-(1,1/5)-(1,3/5)-(1)}) and (\ref{eqn:2nd-deri-(1,1/5)-(1,3/5)-(2)}) yields 
\begin{equation*}
3
\frac
{
\theta^{\prime\prime}
\left[
\begin{array}{c}
1 \\
\frac15
\end{array}
\right]
}
{
\theta
\left[
\begin{array}{c}
1 \\
\frac15
\end{array}
\right]
}
+
3
\frac
{
\theta^{\prime\prime}
\left[
\begin{array}{c}
1 \\
\frac35
\end{array}
\right]
}
{
\theta
\left[
\begin{array}{c}
1 \\
\frac35
\end{array}
\right]
}
-
2
\frac
{
\theta^{\prime\prime \prime}
\left[
\begin{array}{c}
1 \\
1
\end{array}
\right]
}
{
\theta^{\prime}
\left[
\begin{array}{c}
1 \\
1
\end{array}
\right]
}
+
2
\left\{
\frac
{
\theta^{\prime} 
\left[
\begin{array}{c}
1 \\
\frac15
\end{array}
\right]
}
{
\theta
\left[
\begin{array}{c}
1 \\
\frac15
\end{array}
\right]
}
\right\}^2
+
2
\left\{
\frac
{
\theta^{\prime} 
\left[
\begin{array}{c}
1 \\
\frac35
\end{array}
\right]
}
{
\theta
\left[
\begin{array}{c}
1 \\
\frac35
\end{array}
\right]
}
\right\}^2
=0. 
\end{equation*}
The heat equation, Eq. (\ref{eqn:heat}), implies that 
\begin{equation*}
4\pi i 
\frac{d}{d\tau}
\log 
\frac
{
\theta^3
\left[
\begin{array}{c}
1 \\
\frac15
\end{array}
\right]
\theta^3
\left[
\begin{array}{c}
1 \\
\frac15
\end{array}
\right]
}
{
\left\{
\theta^{\prime}
\left[
\begin{array}{c}
1 \\
1
\end{array}
\right]
\right\}^2
}
+
2
\left\{
\frac
{
\theta^{\prime} 
\left[
\begin{array}{c}
1 \\
\frac15
\end{array}
\right]
}
{
\theta
\left[
\begin{array}{c}
1 \\
\frac15
\end{array}
\right]
}
\right\}^2
+
2
\left\{
\frac
{
\theta^{\prime} 
\left[
\begin{array}{c}
1 \\
\frac35
\end{array}
\right]
}
{
\theta
\left[
\begin{array}{c}
1 \\
\frac35
\end{array}
\right]
}
\right\}^2
=0. 
\end{equation*}
Jacobi's triple product identity (\ref{eqn:Jacobi-triple}) yields 
\begin{equation*}
\frac
{
\theta^3
\left[
\begin{array}{c}
1 \\
\frac15
\end{array}
\right]
\theta^3
\left[
\begin{array}{c}
1 \\
\frac15
\end{array}
\right]
}
{
\left\{
\theta^{\prime}
\left[
\begin{array}{c}
1 \\
1
\end{array}
\right]
\right\}^2
}
=
\frac{(\sqrt{5})^3}{4\pi^2}
\frac{\eta^3(5\tau)}{\eta^3(\tau)},
\end{equation*}
which proves the theorem. 
\end{proof}

\begin{corollary}
\label{coro-E_2-P-Q-(1,1/5)-(1,3/5)}
{\it 
For $q\in\mathbb{C}$ with $|q|<1,$ set 
\begin{align*}
P(q)=&\cot \frac{2 \pi}{5} + 4 \sin \frac{4 \pi}{5} \sum_{n=1}^{\infty} (d_{1,5}(n)-d_{4,5}(n)) q^n
-4\sin \frac{2\pi}{5} \sum_{n=1}^{\infty} (d_{2,5}(n)-d_{3,5}(n)) q^n, \\
Q(q)=&\cot \frac{ \pi}{5} + 4 \sin \frac{2 \pi}{5} \sum_{n=1}^{\infty} (d_{1,5}(n)-d_{4,5}(n)) q^n
+4\sin \frac{4\pi}{5} \sum_{n=1}^{\infty} (d_{2,5}(n)-d_{3,5}(n)) q^n. 
\end{align*}
Then, we have 
\begin{equation}
P^2+Q^2=\frac12\left\{-E_2(q)+5E_2(q^5)\right\}. 
\end{equation}
}
\end{corollary}

\begin{theorem}
\label{thm:E_4-E_6-q^5-(1,1/5)-(1,3/5)}
{\it
For $q\in\mathbb{C}$ with $|q|<1,$ set 
\begin{align*}
P(q)=&\cot \frac{2 \pi}{5} + 4 \sin \frac{4 \pi}{5} \sum_{n=1}^{\infty} (d_{1,5}(n)-d_{4,5}(n)) q^n
-4\sin \frac{2\pi}{5} \sum_{n=1}^{\infty} (d_{2,5}(n)-d_{3,5}(n)) q^n, \\
Q(q)=&\cot \frac{ \pi}{5} + 4 \sin \frac{2 \pi}{5} \sum_{n=1}^{\infty} (d_{1,5}(n)-d_{4,5}(n)) q^n
+4\sin \frac{4\pi}{5} \sum_{n=1}^{\infty} (d_{2,5}(n)-d_{3,5}(n)) q^n. 
\end{align*}
Then, we have 
\begin{align*}
E_4(q^5)=&P^4+3 P^3 Q-P^2 Q^2-3 P Q^3+Q^4, \\
E_6(q^5)=&\frac{1}{40} \left(P^2+Q^2\right) \left(41 P^4+198 P^3 Q+154 P^2 Q^2-198 P Q^3+41 Q^4\right).
\end{align*}
}
\end{theorem}

\begin{proof}
Changing $q\rightarrow q^5$ in Ramanujan's system (\ref{eqn:Ramanujan-ODE}), we have 
\begin{equation*}
q\frac{d}{dq} E_2(q^5)=\frac{5}{12}\left\{  \left( E_2(q^5) \right)^2-E_4(q^5)   \right\} , \,\,
q\frac{d}{dq} E_4(q^5)=\frac53 \left\{  E_2(q^5)E_4(q^5) -E_6(q^5)               \right\}. 
\end{equation*}
From Corollary \ref{coro-E_2-P-Q-(1,1/5)-(1,3/5)}, 
we have 
\begin{equation*}
E_2(q^5)=\frac15 (2P^2+2Q^2+R),  \,\,  R=E_2(q),
\end{equation*}
which implies that 
\begin{equation*}
q\frac{d}{dq} E_2(q^5)=
-\frac{7 P^4}{20}-\frac{5 P^3 Q}{4}+\frac{11 P^2 Q^2}{20}   +\frac{5 P Q^3}{4}-\frac{7 Q^4}{20}       +\frac{P^2 R}{15}+\frac{Q^2 R}{15}+\frac{R^2}{60}.
\end{equation*}
Therefore, it follows that 
\begin{equation*}
E_4(q^5)=\left( E_2(q^5) \right)^2-\frac{12}{5} q\frac{d}{dq} E_2(q^5)  =P^4+3 P^3 Q-P^2 Q^2-3 P Q^3+Q^4. 
\end{equation*}
\par
In the same way, 
we note that 
\begin{align*}
q\frac{d}{dq} E_4(q^5)
=&
-\frac{25 P^6}{24}-\frac{25 P^5 Q}{4}-\frac{65 P^4 Q^2}{8}  -\frac{65 P^2 Q^4}{8} +\frac{25 P Q^5}{4} -\frac{25 Q^6}{24}  \\
&\hspace{10mm}
+P^3 Q R-\frac{1}{3} P^2 Q^2 R-P Q^3 R  +\frac{P^4 R}{3}   +\frac{Q^4 R}{3},
\end{align*}
which implies that 
\begin{align*}
E_6(q^5)=& E_2(q^5)E_4(q^5)-\frac35 q\frac{d}{dq} E_4(q^5) \\
=&\frac{1}{40} \left(P^2+Q^2\right) \left(41 P^4+198 P^3 Q+154 P^2 Q^2-198 P Q^3+41 Q^4\right). 
\end{align*}
\end{proof}

\begin{corollary}
\label{coro-E4-E6-level5-q^5-(S,T)}
{\it
For $q\in\mathbb{C}$ with $|q|<1,$ 
set 
\begin{align*}
S(q)=&
\frac
{
\sqrt{250-110 \sqrt{5}}
}
{
5
}
\prod_{n=1}^{\infty}
\frac
{
(1-q^n)^5
}
{
(1-q^{5n})^3
}
\left(
1+
\frac{1-\sqrt{5}}{2} 
q^n
+
q^{2n} 
\right)^5  \\ 
=&
\frac
{
\sqrt{250-110 \sqrt{5}}
}
{
5
}
\left\{
1
-
\frac52(1+\sqrt{5})
\sum_{n=1}^{\infty} (d_{1,5}(n)-d_{4,5}(n)) q^n 
+
\frac52
(7+3 \sqrt{5})
\sum_{n=1}^{\infty} (d_{2,5}(n)-d_{3,5}(n)) q^n
\right\},  
\\
T(q)=&
\frac
{
\sqrt{250+110 \sqrt{5}}
}
{
5
}
\prod_{n=1}^{\infty}
\frac
{
(1-q^n)^5
}
{
(1-q^{5n})^3
}
\left(
1+
\frac{1+\sqrt{5}}{2} 
q^n
+
q^{2n} 
\right)^5  \\
=&
\frac
{
\sqrt{250+110 \sqrt{5}}
}
{
5
}
\left\{
1
+
\frac52(-1+\sqrt{5})
\sum_{n=1}^{\infty} (d_{1,5}(n)-d_{4,5}(n)) q^n 
+
\frac52
(7-3 \sqrt{5})
\sum_{n=1}^{\infty} (d_{2,5}(n)-d_{3,5}(n)) q^n
\right\} . 
\end{align*}
Then we have 
\begin{align*}
E_4(q^5)=&
\frac{S^4-228 S^3 T+494 S^2 T^2+228 S T^3+T^4}{10000}       \\
E_6(q^5)=&-\frac{\left(S^2+T^2\right) \left(S^4+522 S^3 T-10006 S^2 T^2-522 S T^3+T^4\right)}{1000000}, 
\end{align*}
and 
\begin{align*}
q^5(q^5:q^5)_{\infty}^{24}=&
\frac{
(E_4(q^5))^3-(E_6(q^5))^2
}{1728}
=
-\frac{S T \left(S^2+11 S T-T^2\right)^5}{1000000000000} \\
j(q^5)=&-\frac{\left(S^4-228 S^3 T+494 S^2 T^2+228 S T^3+T^4\right)^3}{S T \left(S^2+11 S T-T^2\right)^5}  \\
=&-\frac{\left(f^4+228 f^3+494 f^2-228 f+1\right)^3}{f \left(-f^2+11 f+1\right)^5},    \,\,f=T/S.
\end{align*}
}
\end{corollary}

\begin{proof}
The corollary can be obtained by considering 
\begin{equation*}
S=-3 P+Q, \,\,
T= P+3Q, \,\,\mathrm{and} \,\,f=T/S.
\end{equation*}
\end{proof}

\begin{corollary}
\label{coro-E4-E6-level5-q^5-(V,W)}
{\it
For $q\in\mathbb{C}$ with $|q|<1,$ 
set 
\begin{align*}
V(q)=&\prod_{n=1}^{\infty}
\frac
{
(1-q^n)^5
}
{
(1-q^{5n})^3
}
\left(
1+
\frac{1-\sqrt{5}}{2} 
q^n
+
q^{2n} 
\right)^5   \\
=&
1
-
\frac52(1+\sqrt{5})
\sum_{n=1}^{\infty} (d_{1,5}(n)-d_{4,5}(n)) q^n 
+
\frac52
(7+3 \sqrt{5})
\sum_{n=1}^{\infty} (d_{2,5}(n)-d_{3,5}(n)) q^n,   \\
W(q)=&
\prod_{n=1}^{\infty}
\frac
{
(1-q^n)^5
}
{
(1-q^{5n})^3
}
\left(
1+
\frac{1+\sqrt{5}}{2} 
q^n
+
q^{2n} 
\right)^5  \\
=&
1
+
\frac52(-1+\sqrt{5})
\sum_{n=1}^{\infty} (d_{1,5}(n)-d_{4,5}(n)) q^n 
+
\frac52
(7-3 \sqrt{5})
\sum_{n=1}^{\infty} (d_{2,5}(n)-d_{3,5}(n)) q^n. 
\end{align*}
Then we have 
\begin{align*}
&E_4(q^5)=
\frac{
1
}
{6250}  \times \\
&\times
\Big\{
\left(123-55 \sqrt{5}\right) V^4
+228 \left(11-5 \sqrt{5}\right) V^3 W
+988 V^2 W^2
+228 \left(11+5 \sqrt{5}\right) V W^3
+\left(123+55 \sqrt{5}\right) W^4 
\Big\},
 \\
&E_6(q^5)=
-\frac{\left(\left(11 \sqrt{5}-25\right) V^2-\left(25+11 \sqrt{5}\right) W^2\right) }{1562500}\times  \\
&\times 
\Big\{
\left(55 \sqrt{5}-123\right) V^4-522 \left(5 \sqrt{5}-11\right) V^3 W+20012 V^2 W^2  \\
&\hspace{60mm} +522 \left(11+5 \sqrt{5}\right) V W^3-\left(123+55 \sqrt{5}\right)W^4 
\Big\}
\end{align*}
and 
\begin{align*}
&q^5(q^5;q^5)_{\infty}^{24}
=
-\frac{V W (V-W)^5}{61035156250}\times \\
&\times
\Big\{
\left(75025 \sqrt{5}-167761\right) V^5+10 \left(305 \sqrt{5}-682\right) V^4 W+10 \left(5 \sqrt{5}-11\right) V^3 W^2\\
&\hspace{20mm}+10 \left(11+5 \sqrt{5}\right) V^2 W^3   +10 \left(682+305 \sqrt{5}\right) V W^4+\left(167761+75025 \sqrt{5}\right)W^5
\Big\}.
\end{align*}
}
\end{corollary}

\begin{theorem}
\label{thm:Farkas-Kra-5-(1/5,1)-(3/5,1)}
{\it
For every $\tau\in\mathbb{H}^2,$ we have 
\begin{equation*}
\frac{d}{d \tau} 
\log 
\left(
\frac{\eta(\tau/5)}{\eta(\tau)}
\right)
+
\frac{1}{2\pi i \cdot 3}
\left[
\left\{
\frac
{
\theta^{\prime} 
\left[
\begin{array}{c}
\frac15 \\
1
\end{array}
\right](0,\tau) 
}
{
\theta
\left[
\begin{array}{c}
\frac15 \\
1
\end{array}
\right](0,\tau)
}
\right\}^2
+
\left\{
\frac
{
\theta^{\prime} 
\left[
\begin{array}{c}
\frac35 \\
1
\end{array}
\right](0,\tau) 
}
{
\theta
\left[
\begin{array}{c}
\frac35 \\
1
\end{array}
\right](0,\tau)
}
\right\}^2
\right]=0. 
\end{equation*}
}
\end{theorem}

\begin{proof}
Summing both sides of equations (\ref{eqn:2nd-deri-(1/5,1)-(3/5,1)-(1)}) and (\ref{eqn:2nd-deri-(1/5,1)-(3/5,1)-(2)}) yields 
\begin{equation*}
3
\frac
{
\theta^{\prime\prime}
\left[
\begin{array}{c}
\frac15 \\
1
\end{array}
\right]
}
{
\theta
\left[
\begin{array}{c}
\frac15 \\
1
\end{array}
\right]
}
+
3
\frac
{
\theta^{\prime\prime}
\left[
\begin{array}{c}
\frac35\\
1
\end{array}
\right]
}
{
\theta
\left[
\begin{array}{c}
\frac35 \\
1
\end{array}
\right]
}
-
2
\frac
{
\theta^{\prime\prime \prime}
\left[
\begin{array}{c}
1 \\
1
\end{array}
\right]
}
{
\theta^{\prime}
\left[
\begin{array}{c}
1 \\
1
\end{array}
\right]
}
+
2
\left\{
\frac
{
\theta^{\prime} 
\left[
\begin{array}{c}
\frac15 \\
1
\end{array}
\right]
}
{
\theta
\left[
\begin{array}{c}
\frac15 \\
1
\end{array}
\right]
}
\right\}^2
+
2
\left\{
\frac
{
\theta^{\prime} 
\left[
\begin{array}{c}
\frac35 \\
1
\end{array}
\right]
}
{
\theta
\left[
\begin{array}{c}
\frac35 \\
1
\end{array}
\right]
}
\right\}^2
=0. 
\end{equation*}
The heat equation (\ref{eqn:heat}) implies that 
\begin{equation*}
4\pi i 
\frac{d}{d\tau}
\log 
\frac
{
\theta^3
\left[
\begin{array}{c}
\frac15 \\
1
\end{array}
\right]
\theta^3
\left[
\begin{array}{c}
\frac35 \\
1
\end{array}
\right]
}
{
\left\{
\theta^{\prime}
\left[
\begin{array}{c}
1 \\
1
\end{array}
\right]
\right\}^2
}
+
2
\left\{
\frac
{
\theta^{\prime} 
\left[
\begin{array}{c}
\frac15 \\
1
\end{array}
\right]
}
{
\theta
\left[
\begin{array}{c}
\frac15 \\
1
\end{array}
\right]
}
\right\}^2
+
2
\left\{
\frac
{
\theta^{\prime} 
\left[
\begin{array}{c}
\frac35 \\
1
\end{array}
\right]
}
{
\theta
\left[
\begin{array}{c}
\frac35 \\
1
\end{array}
\right]
}
\right\}^2
=0. 
\end{equation*}
Jacobi's triple product identity (\ref{eqn:Jacobi-triple}) yields 
\begin{equation*}
\frac
{
\theta^3
\left[
\begin{array}{c}
\frac15 \\
1
\end{array}
\right]
\theta^3
\left[
\begin{array}{c}
\frac35 \\
1
\end{array}
\right]
}
{
\left\{
\theta^{\prime}
\left[
\begin{array}{c}
1 \\
1
\end{array}
\right]
\right\}^2
}
=
\frac{\zeta_5^3}{4\pi^2}
\frac{\eta^3(\tau/5)}{\eta^3(\tau)},
\end{equation*}
which proves the theorem. 
\end{proof}

\begin{corollary}
\label{coro-E_2-P-Q-(1/5,1)-(3/5,1)}
{\it
For $q\in\mathbb{C}$ with $|q|<1,$ set
\begin{align*}
P(q)=&1+ 10 \sum_{n=1}^{\infty} (d_{2,5}(n)-d_{3,5}(n)) q^n, \,\,
Q(q)=3 + 10\sum_{n=1}^{\infty} (d_{1,5}(n)-d_{4,5}(n)) q^n. 
\end{align*}
Then, we have 
\begin{equation}
P^2+Q^2=\frac52\left\{ -E_2(q)+5E_2(q^5) \right\}.
\end{equation}
}
\end{corollary}

\begin{theorem}
\label{thm:E_4-E_6-q-(1/5,1)-(3/5,1)}
{\it
For $q\in\mathbb{C}$ with $|q|<1,$ set 
\begin{align*}
P(q)=&1+ 10 \sum_{n=1}^{\infty} (d_{2,5}(n)-d_{3,5}(n)) q^n, \,\,
Q(q)=3 + 10\sum_{n=1}^{\infty} (d_{1,5}(n)-d_{4,5}(n)) q^n. 
\end{align*}
Then, we have 
\begin{align*}
E_4(q)=&P^4+3 P^3 Q-P^2 Q^2-3 P Q^3+Q^4, \\
E_6(q)=&-\frac{1}{40} \left(P^2+Q^2\right) \left(41 P^4+198 P^3 Q+154 P^2 Q^2-198 P Q^3+41 Q^4\right),
\end{align*}
and
\begin{align*}
1728 q(q;q)_{\infty}^{24}
=&
(E_4(q))^3-(E_6(q))^2
=-\frac{27 (3 P-Q) (P+3 Q) \left(P^2+4 P Q-Q^2\right)^5}{1600}, \\ 
j(q)=&
\frac{ 1728  (E_4(q))^3 }{ (E_4(q))^3-(E_6(q))^2 } 
=
-\frac{102400 \left(P^4+3 P^3 Q-P^2 Q^2-3 P Q^3+Q^4\right)^3}{(3 P-Q) (P+3 Q) \left(P^2+4 P Q-Q^2\right)^5} \\
=&
\frac{\left(g^4+228 g^3+494 g^2-228 g+1\right)^3}{g \left(g^2-11 g-1\right)^5}, 
\end{align*}
where 
\begin{equation*}
g(q)=\frac{3P-Q}{P+3Q}=
-
q
\prod_{n=1}^{\infty}\frac{(1-q^{5n-1})^5(1-q^{5n-4})^5  }{(1-q^{5n-2})^5(1-q^{5n-3})^5  }. 
\end{equation*}
}
\end{theorem}

\begin{proof}
From Corollary \ref{coro-E_2-P-Q-(1/5,1)-(3/5,1)}, we have 
\begin{equation*}
E_2(q)=\frac{ -2P^2-2Q^2+25 R }{5},         \,\,\,R=E_2(q^5),
\end{equation*}
which implies that 
\begin{equation*}
q\frac{d}{dq} E_2(q)=
-\frac{7 P^4}{100}-\frac{P^3 Q}{4}+\frac{11 P^2 Q^2}{100} +\frac{P Q^3}{4}-\frac{7 Q^4}{100}
-\frac{P^2 R}{3}-\frac{Q^2 R}{3}+\frac{25 R^2}{12}. 
\end{equation*}
By Ramanujan's system (\ref{eqn:Ramanujan-ODE}),  we obtain 
\begin{equation*}
E_4(q)=\left\{  E_2(q)  \right\}^2-12 q\frac{d}{dq} E_2(q)
=P^4+3 P^3 Q-P^2 Q^2-3 P Q^3+Q^4. 
\end{equation*}
In the same way, 
we note that 
\begin{align*}
q\frac{d}{dq} E_4(q)=&
\frac{5 P^6}{24}+\frac{5 P^5 Q}{4}+\frac{13 P^4 Q^2}{8}       +\frac{13 P^2 Q^4}{8} -\frac{5 P Q^5}{4} +\frac{5 Q^6}{24}  \\
&\hspace{20mm} -\frac{5}{3} P^2 Q^2 R+5 P^3 Q R-5 P Q^3 R   +\frac{5 P^4 R}{3}+\frac{5 Q^4 R}{3},
\end{align*}
which implies that 
\begin{align*}
E_6(q)=&E_2(q) E_4(q)-3  q\frac{d}{dq} E_4(q)  \\
=&-\frac{1}{40} \left(P^2+Q^2\right) \left(41 P^4+198 P^3 Q+154 P^2 Q^2-198 P Q^3+41 Q^4\right). 
\end{align*}
\end{proof}

\begin{corollary}
\label{coro-E4-E6-level5-q-(S,T)}
{\it
For $q\in\mathbb{C}$ with $|q|<1,$ 
set 
\begin{align*}
S(q)=&
q
\prod_{n=1}^{\infty} 
\frac
{
(1-q^n)^2
}
{
(1-q^{5n-2})^5(1-q^{5n-3})^5    
}
=
\sum_{n=1}^{\infty} (d_{1,5}(n)-d_{4,5}(n)) q^n
-3
\sum_{n=1}^{\infty} (d_{2,5}(n)-d_{3,5}(n)) q^n,  
\end{align*}
and 
\begin{align*}
T(q)=&
\prod_{n=1}^{\infty} 
\frac
{
(1-q^n)^2
}
{
(1-q^{5n-1})^5(1-q^{5n-4})^5    
}
=
1+
3
\sum_{n=1}^{\infty} (d_{1,5}(n)-d_{4,5}(n)) q^n
+
\sum_{n=1}^{\infty} (d_{2,5}(n)-d_{3,5}(n)) q^n.
\end{align*}
Then, we have 
\begin{align*}
E_4(q)=&
S^4-228 S^3 T+494 S^2 T^2+228 S T^3+T^4,  \\
E_6(q)=&
\left(S^2+T^2\right) \left(S^4+522 S^3 T-10006 S^2 T^2-522 S T^3+T^4\right), 
\end{align*}
and
\begin{align*}
q(q;q)_{\infty}=&-S T \left(S^2+11 S T-T^2\right)^5  \\
j(q)=&
-\frac{\left(S^4-228 S^3 T+494 S^2 T^2+228 S T^3+T^4\right)^3}{S T \left(S^2+11 S T-T^2\right)^5}
=
-\frac{\left(g^4-228 g^3+494 g^2+228 g+1\right)^3}{g \left(g^2+11 g-1\right)^5},  
\end{align*}
where 
\begin{equation*}
g=S/T
=
q
\prod_{n=1}^{\infty}\frac{(1-q^{5n-1})^5(1-q^{5n-4})^5  }{(1-q^{5n-2})^5(1-q^{5n-3})^5  }. 
\end{equation*}
}
\end{corollary}

\subsubsection*{Remark}

Corollary \ref{coro-E_2-P-Q-(1/5,1)-(3/5,1)} and Theorem \ref{thm:E_4-E_6-q-(1/5,1)-(3/5,1)} enable us to consider the extension of the differential fields, 
\begin{equation*}
\mathbb{Q}\left( E_2(q), E_4(q), E_6(q)  \right) \subset \mathbb{Q}\left( P(q), Q(q), E_2(q)  \right),
\end{equation*}
where 
\begin{align*}
P(q)=&1+ 10 \sum_{n=1}^{\infty} (d_{2,5}(n)-d_{3,5}(n)) q^n, \,\,
Q(q)=3 + 10\sum_{n=1}^{\infty} (d_{1,5}(n)-d_{4,5}(n)) q^n. 
\end{align*}

\section{Riccati equations satisfied by level 5 modular forms}
\label{sec:Riccati-level-5}

\subsection{Derivation of Riccati equation (1)}

\begin{proposition}
\label{prop:2nd-order-theta(1,1/5)-(1,3/5)}
{\it
For every $\tau\in\mathbb{H}^2,$ we have 
\begin{equation*}
\frac
{
\theta^{\prime \prime}
\left[
\begin{array}{c}
1 \\
\frac15
\end{array}
\right] 
}
{
\theta
\left[
\begin{array}{c}
1 \\
\frac15
\end{array}
\right] 
}
-
\frac
{
\theta^{\prime \prime}
\left[
\begin{array}{c}
1 \\
\frac35
\end{array}
\right] 
}
{
\theta
\left[
\begin{array}{c}
1 \\
\frac35
\end{array}
\right] 
}
=
\frac
{
\left\{
\theta^{\prime}
\left[
\begin{array}{c}
1 \\
1
\end{array}
\right] 
\right\}^2
}
{
100
\theta^6
\left[
\begin{array}{c}
1 \\
\frac15
\end{array}
\right] 
\theta^6
\left[
\begin{array}{c}
1 \\
\frac35
\end{array}
\right] 
}
\left\{
-8
\theta^{10}
\left[
\begin{array}{c}
1 \\
\frac15
\end{array}
\right] 
+88
\theta^5
\left[
\begin{array}{c}
1 \\
\frac15
\end{array}
\right] 
\theta^5
\left[
\begin{array}{c}
1 \\
\frac35
\end{array}
\right] 
+8
\theta^{10}
\left[
\begin{array}{c}
1 \\
\frac35
\end{array}
\right] 
\right\}. 
\end{equation*}
}
\end{proposition}

\begin{proof}
The proposition can be obtained by 
considering the derivative formulas (\ref{eqn:analogue-Jacobi-(1,1/5)}) and (\ref{eqn:analogue-Jacobi-(1,3/5)}) and 
by comparing the coefficients of term $z^2$ in the equation (\ref{eqn:level5-(1,1/5),(1,3/5)}). 
\end{proof}

\begin{theorem}
\label{thm:q-Riccati-diff-eq-modulus5-(1)}
{\it
For every $\tau\in\mathbb{H}^2,$ set
\begin{equation*}
W=
\frac{
\theta^5
\left[
\begin{array}{c}
1 \\
\frac15
\end{array}
\right](0,\tau)
}
{
\theta^5
\left[
\begin{array}{c}
1 \\
\frac35
\end{array}
\right](0,\tau)
}=\frac{11+5\sqrt{5}}{2}
\prod_{n=1}^{\infty}
\frac{ \left(1+\frac{1+\sqrt{5}}{2}q^n+q^{2n} \right)^5 }{ \left(1+\frac{1-\sqrt{5}}{2}q^n+q^{2n} \right)^5 },
 \,\,q=\exp(2\pi i \tau).
\end{equation*}
$W$ then satisfies the following Riccati equation:
\begin{equation}
\label{eqn:Riccati-q-diff-eq-modulus5-(1)}
q
\frac{d}{dq} W
=
\frac{1}{(\sqrt{5})^3}
\frac{(q;q)_{\infty}^5}{(q^5;q^5)_{\infty}}
\left(
W^2-11W-1
\right). 
\end{equation}
}
\end{theorem}

\begin{proof}
The heat equation (\ref{eqn:heat}) and 
Proposition \ref{prop:2nd-order-theta(1,1/5)-(1,3/5)}
imply that 
\begin{align*}
&
4\pi i 
\frac
{
d
}
{
d\tau
}
\log 
\theta
\left[
\begin{array}{c}
1 \\
\frac15
\end{array}
\right] 
-
4\pi i 
\frac
{
d
}
{
d\tau
}
\log 
\theta
\left[
\begin{array}{c}
1 \\
\frac35
\end{array}
\right]   \\
=&
\frac
{
\left\{
\theta^{\prime}
\left[
\begin{array}{c}
1 \\
1
\end{array}
\right] 
\right\}^2
}
{
100
\theta
\left[
\begin{array}{c}
1 \\
\frac15
\end{array}
\right] 
\theta
\left[
\begin{array}{c}
1 \\
\frac35
\end{array}
\right] 
}
\frac
{
\left\{
-8
\theta^{10}
\left[
\begin{array}{c}
1 \\
\frac15
\end{array}
\right] 
+88
\theta^5
\left[
\begin{array}{c}
1 \\
\frac15
\end{array}
\right] 
\theta^5
\left[
\begin{array}{c}
1 \\
\frac35
\end{array}
\right] 
+8
\theta^{10}
\left[
\begin{array}{c}
1 \\
\frac35
\end{array}
\right] 
\right\}
}
{
\theta^5
\left[
\begin{array}{c}
1 \\
\frac15
\end{array}
\right] 
\theta^5
\left[
\begin{array}{c}
1 \\
\frac35
\end{array}
\right] 
},
\end{align*}
which shows that 
\begin{align*}
&
4\pi i 
\frac
{
d
}
{
d \tau
}
\log 
\theta^5
\left[
\begin{array}{c}
1 \\
\frac15
\end{array}
\right] 
-
4\pi i 
\frac
{
d
}
{
d\tau
}
\log 
\theta^5
\left[
\begin{array}{c}
1 \\
\frac35
\end{array}
\right]   \\
=&
\frac
{
\left\{
\theta^{\prime}
\left[
\begin{array}{c}
1 \\
1
\end{array}
\right] 
\right\}^2
}
{
20
\theta
\left[
\begin{array}{c}
1 \\
\frac15
\end{array}
\right] 
\theta
\left[
\begin{array}{c}
1 \\
\frac35
\end{array}
\right] 
}
\frac
{
\left\{
-8
\theta^{10}
\left[
\begin{array}{c}
1 \\
\frac15
\end{array}
\right] 
+88
\theta^5
\left[
\begin{array}{c}
1 \\
\frac15
\end{array}
\right] 
\theta^5
\left[
\begin{array}{c}
1 \\
\frac35
\end{array}
\right] 
+8
\theta^{10}
\left[
\begin{array}{c}
1 \\
\frac35
\end{array}
\right] 
\right\}
}
{
\theta^5
\left[
\begin{array}{c}
1 \\
\frac15
\end{array}
\right] 
\theta^5
\left[
\begin{array}{c}
1 \\
\frac35
\end{array}
\right] 
}.
\end{align*}
\par
Jacobi's triple product identity (\ref{eqn:Jacobi-triple}) yields 
\begin{equation*}
\frac
{
\left\{
\theta^{\prime}
\left[
\begin{array}{c}
1 \\
1
\end{array}
\right] 
\right\}^2
}
{
\theta
\left[
\begin{array}{c}
1 \\
\frac15
\end{array}
\right] 
\theta
\left[
\begin{array}{c}
1 \\
\frac35
\end{array}
\right] 
}
=\frac{4\pi^2}{\sqrt{5}}
\prod_{n=1}^{\infty} \frac{(1-q^n)^5}{(1-q^{5n})}
=\frac{4\pi^2}{\sqrt{5}}
\frac{\eta(\tau)^5}{\eta(5\tau)}. 
\end{equation*}
\par
By setting 
\begin{equation*}
(X,Y)=
\left(
\theta^5
\left[
\begin{array}{c}
1 \\
\frac15
\end{array}
\right],  
\theta^5
\left[
\begin{array}{c}
1 \\
\frac35
\end{array}
\right] 
\right), 
\end{equation*}
we obtain 
\begin{equation*}
\frac{d X}{d\tau} Y
-
X \frac{d Y}{d\tau}
=
\frac{2\pi i}{(\sqrt{5})^3}
\frac{\eta^5(\tau)}{\eta(5\tau)}
\left(X^2-11XY-Y^2\right),
\end{equation*}
which implies that 
\begin{equation*}
\frac{d}{d\tau} W
=
\frac{d}{d\tau}
\left(
\frac{X}{Y}
\right)
=
\frac
{
\displaystyle \frac{dX}{d\tau}Y-X\frac{dY}{d\tau}
}
{
Y^2
}
=
\frac{2 \pi i}{(\sqrt{5})^3}
\frac{\eta^5(\tau)}{\eta(5\tau)}
\left(
W^2-11W-1
\right).
\end{equation*}
The theorem can be obtained by considering that 
$
\displaystyle
\frac{d}{d\tau}=\frac{dq}{d\tau}\frac{d}{dq}=2\pi i q \frac{d}{dq}. 
$
\end{proof}

\subsection{Derivation of a Riccati equation (2)}

\begin{proposition}
\label{prop:2nd-order-theta(1/5,1)-(3/5,1)}
{\it
For every $\tau\in\mathbb{H}^2,$ we have 
\begin{equation*}
\frac
{
\theta^{\prime \prime}
\left[
\begin{array}{c}
\frac35 \\
1
\end{array}
\right] 
}
{
\theta
\left[
\begin{array}{c}
\frac35 \\
1
\end{array}
\right] 
}
-
\frac
{
\theta^{\prime \prime}
\left[
\begin{array}{c}
\frac15 \\
1
\end{array}
\right] 
}
{
\theta
\left[
\begin{array}{c}
\frac15 \\
1
\end{array}
\right] 
}
=
\frac{2\zeta_5}{25}
\frac
{
\left\{
\theta^{\prime}
\left[
\begin{array}{c}
1 \\
1
\end{array}
\right] 
\right\}^2
}
{
\theta^6
\left[
\begin{array}{c}
\frac15 \\
1
\end{array}
\right] 
\theta^6
\left[
\begin{array}{c}
\frac35 \\
1
\end{array}
\right] 
}
\left\{
\theta^{10}
\left[
\begin{array}{c}
\frac15 \\
1
\end{array}
\right] 
+11
\theta^5
\left[
\begin{array}{c}
\frac15 \\
1
\end{array}
\right] 
\theta^5
\left[
\begin{array}{c}
\frac35 \\
1
\end{array}
\right] 
-
\theta^{10}
\left[
\begin{array}{c}
\frac35 \\
1
\end{array}
\right] 
\right\}. 
\end{equation*}
}
\end{proposition}

\begin{proof}
The proposition can be obtained by 
considering the derivative formulas (\ref{eqn:analogue-Jacobi-(1/5,1)}) and (\ref{eqn:analogue-Jacobi-(3/5,1)}) and by 
comparing the coefficients of the term $z^2$ in equation (\ref{eqn:level5-(1/5,1),(3/5,1)}). 
\end{proof}

\begin{theorem}
\label{thm:q-Riccati-diff-eq-modulus5-(2)}
{\it
For every $\tau\in\mathbb{H}^2,$ set
\begin{equation*}
W=
\frac{
\theta^5
\left[
\begin{array}{c}
\frac35 \\
1
\end{array}
\right](0, 5\tau)
}
{
\theta^5
\left[
\begin{array}{c}
\frac15 \\
1
\end{array}
\right](0, 5\tau)
}=
-q
\prod_{n=1}^{\infty} 
\frac{ (1-q^{5n-1} )^5 (1-q^{5n-4} )^5 }{ (1-q^{5n-2} )^5 (1-q^{5n-3} )^5 },
 \,\,q=\exp(2\pi i \tau).
\end{equation*}
$W$ then satisfies the following Riccati equation:
\begin{equation}
\label{eqn:Riccati-q-diff-eq-modulus5-(2)}
q
\frac{d}{dq} W
=
q\frac{(q^5;q^5)_{\infty}^5}{(q ;q)_{\infty}}
\left(
W^2-11W-1
\right). 
\end{equation}
}
\end{theorem}

\begin{proof}
The heat equation (\ref{eqn:heat}) and 
Proposition \ref{prop:2nd-order-theta(1/5,1)-(3/5,1)}
imply that 
\begin{align*}
&
4\pi i 
\frac
{
d
}
{
d\tau
}
\log 
\theta
\left[
\begin{array}{c}
\frac35 \\
1
\end{array}
\right] 
-
4\pi i 
\frac
{
d
}
{
d\tau
}
\log 
\theta
\left[
\begin{array}{c}
\frac15 \\
1
\end{array}
\right]   \\
=&
-
\frac{2\zeta_5}{25}
\frac
{
\left\{
\theta^{\prime}
\left[
\begin{array}{c}
1 \\
1
\end{array}
\right] 
\right\}^2
}
{
\theta
\left[
\begin{array}{c}
\frac35 \\
1
\end{array}
\right] 
\theta
\left[
\begin{array}{c}
\frac15 \\
1
\end{array}
\right] 
}
\frac
{
\left\{
\theta^{10}
\left[
\begin{array}{c}
\frac35 \\
1
\end{array}
\right] 
-11
\theta^5
\left[
\begin{array}{c}
\frac35 \\
1
\end{array}
\right] 
\theta^5
\left[
\begin{array}{c}
\frac15 \\
1
\end{array}
\right] 
-
\theta^{10}
\left[
\begin{array}{c}
\frac15 \\
1
\end{array}
\right] 
\right\}
}
{
\theta^5
\left[
\begin{array}{c}
\frac35 \\
1
\end{array}
\right] 
\theta^5
\left[
\begin{array}{c}
\frac15 \\
1
\end{array}
\right] 
},
\end{align*}
which shows that 
\begin{align*}
&
4\pi i 
\frac
{
d
}
{
d \tau
}
\log 
\theta^5
\left[
\begin{array}{c}
\frac35 \\
1
\end{array}
\right] 
-
4\pi i 
\frac
{
d
}
{
d\tau
}
\log 
\theta^5
\left[
\begin{array}{c}
\frac15 \\
1
\end{array}
\right]   \\
=&
-
\frac{2\zeta_5}{5}
\frac
{
\left\{
\theta^{\prime}
\left[
\begin{array}{c}
1 \\
1
\end{array}
\right] 
\right\}^2
}
{
\theta
\left[
\begin{array}{c}
\frac35 \\
1
\end{array}
\right] 
\theta
\left[
\begin{array}{c}
\frac15 \\
1
\end{array}
\right] 
}
\frac
{
\left\{
\theta^{10}
\left[
\begin{array}{c}
\frac35 \\
1
\end{array}
\right] 
-
11
\theta^5
\left[
\begin{array}{c}
\frac35 \\
1
\end{array}
\right] 
\theta^5
\left[
\begin{array}{c}
\frac15 \\
1
\end{array}
\right] 
-
\theta^{10}
\left[
\begin{array}{c}
\frac15 \\
1
\end{array}
\right] 
\right\}
}
{
\theta^5
\left[
\begin{array}{c}
\frac35 \\
1
\end{array}
\right] 
\theta^5
\left[
\begin{array}{c}
\frac15 \\
1
\end{array}
\right] 
}.
\end{align*}
\par
Jacobi's triple product identity (\ref{eqn:Jacobi-triple}) yields 
\begin{equation*}
\frac
{
\left\{
\theta^{\prime}
\left[
\begin{array}{c}
1 \\
1
\end{array}
\right] 
\right\}^2
}
{
\theta
\left[
\begin{array}{c}
\frac35 \\
1
\end{array}
\right] 
\theta
\left[
\begin{array}{c}
\frac15 \\
1
\end{array}
\right] 
}
=\frac{4\pi^2}{ \zeta_5}
y\prod_{n=1}^{\infty} \frac{(1-y^{5n})^5}{(1-y^{n})} 
=\frac{4\pi^2}{ \zeta_5}\frac{\eta^5(\tau)}{\eta(\tau/5)}, \,\,y=\exp(2 \pi i \tau/5). 
\end{equation*}
\par
By setting 
\begin{equation*}
(X,Y)=
\left(
\theta^5
\left[
\begin{array}{c}
\frac15 \\
1
\end{array}
\right](0,\tau),   
\theta^5
\left[
\begin{array}{c}
\frac35 \\
1
\end{array}
\right](0,\tau) 
\right), 
\end{equation*}
we obtain 
\begin{equation*}
y \frac{d}{d y} W
=
y
\frac{d}{d y}
\left(
\frac{Y}{X}
\right)
=
y \frac{(y^5;y^5)^5_{\infty}}{(y;y)_{\infty}} 
\left(
W^2-11W-1
\right).
\end{equation*}
The theorem can be obtained by changing $\tau\longrightarrow 5\tau.$ 
\end{proof}



\begin{thebibliography}{}
%
%



\bibitem{Apostol}
T. M. Apostol, 
{\it Modular functions and Dirichlet series in number theory,} 
Second edition. Grad. Texts in Math. {\bf 41} Springer-Verlag, New York, 1990. 




\bibitem{Chazy}
J. Chazy, 
{\it Sur les \'equations diff\'erentielles du troisi\'eme ordre et d\'ordre sup\'erieur dont l\'intégrale g\'en\'erale a ses points critiques fixes,}
Acta Math. {\bf 34} (1911), 317-385. 



\bibitem{Cooper}
S. Cooper, 
{\it
Ramanujan's theta functions.}
Springer, Cham, 2017.

\bibitem{Darboux}
G. Darboux, 
{\it Sur la th\'eorie des coordonn\'ees curvilignes et les syst\'emes orthogonaux,} 
Ann. Ecole Normale Sup\'erieure {\bf 7}  (1878) 101-150 





\bibitem{Farkas-Kra}
H. M. Farkas and I. Kra, 
{\it Theta constants, Riemann surfaces and the modular group,} 
AMS Grad. Studies in Math. {\bf 37} (2001). 



\bibitem{Halphen}
G. Halphen, 
{\it Sur une system d\'equations differ entielles,} 
C. R. Acad. Sci., Paris {\bf 92} (1881), 1101-1103. 

\bibitem{Hermite}
C. Hermite, Oeuvres, Vol. II, Gauthier-Villars, Paris, 1908. 
{\it Note sur la th\'eorie des fonctions elliptiques,} pp125-238.


\bibitem{Huber}
T. Huber, 
{\it Differential equations for cubic theta functions,} 
Int. J. Number Theory {\bf  71}  (2011), 1945-1957. 




\bibitem{Matsuda}
K. Matsuda, 
{\it  
Differential equations satisfied by $a(q)=\sum_{m,n\in\mathbb{Z}} q^{m^2+mn+n^2}$
},
arXiv:1609.07481


\bibitem{Ohyama0}
Y. Ohyama, 
{\it Differential relations of theta functions. }
Osaka J. Math.  32  (1995), 431-450. 



\bibitem{Ramanujan}
S. Ramanujan, 
{\it On certain arithmetical functions,} 
Trans. Cambridge Philos. Soc. {\bf 22}, no. 9, (1916) 159-184. 






\bibitem{WW}
E. T. Whittaker, and G. N. Watson, 
{\it A course of modern analysis. An introduction to the general theory of infinite processes and of analytic functions; with an account of the principal transcendental functions,}
Fourth edition. Reprinted Cambridge University Press, New York 1962

 
\end{thebibliography}


\end{document}